\documentclass[secthm]{elsart}
\usepackage{amsmath,amssymb,pstricks} 

\newtheorem{example}{Example}

\newcommand{\ees}{\mathfrak{s}}
\newcommand{\eel}{\mathfrak{w}}
\newcommand{\eeg}{\Gamma_{X}}
\newcommand{\eeL}{\mathfrak{L}}
\newcommand{\eeR}{\mathfrak{R}}
\newcommand{\cutie}{D(\vec{\omega})=[a_{j}(m_{j});b_{i}(n_{i})]_{q,t}}
\newcommand{\ve}{\varepsilon}

\theoremstyle{remark}

\newenvironment{proof}{\medskip \noindent \textbf{Proof.}\ }{\smallskip}
\newenvironment{notation}{\medskip \noindent \textbf{Notation.}\ }{\smallskip}

\begin{document}

\begin{frontmatter}

\title{Enhanced negative type for finite metric trees}

\author{Ian Doust, Anthony Weston}
\address{School of Mathematics and Statistics,
University of New South Wales, Sydney, New South Wales 2052, Australia}
\ead{i.doust@unsw.edu.au}
\address{Department of Mathematics and Statistics, Canisius College,
Buffalo, New York 14208, United States of America}
\ead{westona@canisius.edu}
\dedicated{Dedicated to Bernard Joseph Weston (and the spirit of the Eureka Stockade) \\
on the occasion of his 80th birthday --- 3 June 2007.}

\begin{abstract}
A finite metric tree is a finite connected graph that has no cycles, endowed with an edge weighted
path metric. Finite metric trees are known to have strict $1$-negative type.
In this paper we introduce a new family of inequalities (\ref{ONE}) that encode the best
possible quantification of the \textit{strictness} of the non trivial $1$-negative type inequalities
for finite metric trees.
These inequalities are
sufficiently strong to imply that any given finite metric tree $(T,d)$ must have strict $p$-negative type
for all $p$ in an open interval $(1-\zeta,1+\zeta)$, where $\zeta > 0$ may be chosen so as to depend \textit{only} upon the
unordered distribution of edge weights that determine the path metric $d$ on $T$.
In particular, if the edges of the tree are not weighted, then it follows
that $\zeta$ depends only upon the number of vertices in the tree.

We also give an example of an infinite metric
tree that has strict $1$-negative type but does not have $p$-negative type for any $p > 1$.
This shows that the maximal $p$-negative type of a metric space can be strict.
\end{abstract}

\end{frontmatter}

\section{Introduction and Synopsis}
The study of trees as mathematical objects was initiated by Cayley \cite{CAY} who enumerated the isomers
of the saturated hydrocarbons $C_{n}H_{2n+2}$. For example, an application of Cayley's formula shows that
the number of isomers of the paraffin $C_{13}H_{28}$ is $802$. More recently,
mathematical studies of finite metric trees have proliferated due to myriad applications in areas
as diverse as evolutionary biology and theoretical computer science. Some examples of publications which
highlight this point include Weber \textit{et al}.\ \cite{WOS}, Ailon and Charikar \cite{AC},
Semple and Steel \cite{SS}, Fakcharoenphol \textit{et al}.\ \cite{FRT}, Charikar \textit{et al}.\ \cite{CCGGP},
and Bartal \cite{BART}.

Works such as those cited above illustrate two of the major themes of study pertaining to metric trees.
One is to try to reconstruct
metric trees from data such as DNA or protein sequences. This is the realm of so called phylogenetic
tree reconstruction or (more
generally) numerical taxonomy. The second major theme, driven by algorithmic considerations in computer science, is
to approximate finite metrics by (small numbers of) tree metrics. The importance of finite metric trees in this context
is that they are well suited to algorithms and can serve to help greatly reduce the computational hardness of
certain optimization problems.

In this paper we focus on one particular aspect of the non linear geometry of finite metric trees;
namely, strict $p$-negative type. (See Definition \ref{NTGRDEF}.)
The $p$-negative type inequalities arose classically in studies of isometric
embeddings and remain objects of intense research scrutiny in areas ranging from functional analysis to theoretical
computer science. The monographs of Wells and Williams \cite{WW}, and Deza and Laurent \cite{DL}, illustrate
a variety of classical and contemporary applications of inequalities of $p$-negative type.
See also the comments in Section $2$ of this paper.

Hjorth \textit{et al}.\ \cite{HLM} have shown that finite metric trees have strict $1$-negative type.
In this paper we determine that a new and substantially stronger family of inequalities
hold for finite metric trees. Namely, as we show in Theorems \ref{MAIN} and \ref{REFORM},
given a finite metric tree $(T,d)$, there is a maximal constant $\Gamma_{T} > 0$
so that for all natural numbers $n \geq 2$, all finite subsets
$\{x_{1}, \ldots , x_{n} \} \subseteq T$, and all choices of real numbers $\eta_{1}, \ldots , \eta_{n}$ with
$\eta_{1} + \cdots + \eta_{n} = 0$, we have:
\begin{eqnarray}\label{ONE}
\frac{\Gamma_{T}}{2} \cdot \Biggl( \sum\limits_{\ell=1}^{n} |\eta_{\ell}| \Biggl)^{2} +
\sum\limits_{1 \leq i,j \leq n} d(x_{i},x_{j}) \eta_{i} \eta_{j} & \leq & 0.
\end{eqnarray}
We call the maximal constant $\Gamma_{T}$
appearing in (\ref{ONE}) the $1$-negative type gap of $(T,d)$.
Theorem \ref{MAIN} includes a characterization of equality in the inequalities (\ref{ONE}).
Remark \ref{altchar} then indicates an alternative
and more direct characterization of equality in the inequalities (\ref{ONE}).
In Corollary \ref{GAPFORMULA} we compute a closed formula for the exact value of $\Gamma_{T}$ and thereby show that it
depends only upon the tree's unordered distribution of edge weights. Indeed,
\[
\Gamma_{T} = \biggl\{ \sum\limits_{(x,y) \in E(T)} d(x,y)^{-1} \biggl\}^{-1}
\]
where the sum is taken over the set of all (unordered) edges $e=(x,y)$ in $T$.

The inequalities (\ref{ONE}) are particularly strong. They imply, for example, that there is an $\zeta > 0$ so that the
finite metric tree $(T,d)$ has strict $p$-negative type for all $p \in (1-\zeta,1+\zeta)$. Moreover, due to the
universality of $\Gamma_{T}$, $\zeta$ can be chosen so that it depends only upon the tree's unordered distribution
of edge weights. This is done in Theorem \ref{LOWER}. So, in this context ($p=1$), the strict negative type of
finite metric trees is seen to persist on open intervals.
The same cannot be said of infinite metric trees as demonstrated by the infinite
necklace tree $(Y,d)$ which is described in Example \ref{STAR2} and Theorem \ref{STRICTMAX}. The necklace $(Y,d)$
has strict $1$-negative type but does not have $p$-negative type for any $p>1$. This example also shows that the
maximal $p$-negative type of a metric space can be strict. It is an open question whether
the maximal $p$-negative type of a finite metric space can be strict.

In this paper we adopt a vicarious (rather than direct) approach to $p$-negative type by
choosing to work with the equivalent notion of generalized roundness-$p$. (See Definition \ref{NTGRDEF}.)
We utilize this approach due to the fact that the geometric notion of generalized roundness-$p$
seems much more well suited to the analysis of highly symmetric objects (such as metric trees)
than the more analytic notion of $p$-negative type. The results of this paper validate this approach.

The remainder of this paper is structured as follows. Section $2$ discusses all relevant background
material on $p$-negative type and generalized roundness-$p$. The known equivalence
of these two notions is expressed in Theorem \ref{GRUNT}. This equivalence constitutes the primary theoretical
tool of the entire paper. Section $2$ also introduces the $p$-negative type gap $\Gamma_{X,p}$ of a metric
space $(X,d)$. Section $3$ develops some basic facts pertaining to finite metric trees $(T,d)$ and takes some
initial steps towards determining the maximal constant $\Gamma_{T}=\Gamma_{T,1}$ that appears in (\ref{ONE}).
Section $4$ completes this process, via Lagrange's (Multiplier) Theorem, and this leads into the derivations
of Theorems \ref{MAIN} and \ref{REFORM} (as discussed above). Section $4$ also introduces the ``generic algorithm''
(Definition \ref{GENALG}) which provides a means to characterize equality in the inequalities (\ref{ONE}) above.
Section $5$ develops applications of the inequalities (\ref{ONE}) such as Theorem \ref{LOWER} (which determines
lower bounds on the $p$-negative type of finite metric trees) and
Theorem \ref{STRICTMAX} (which gives properties of the infinite necklace $(Y,d)$).
Throughout this paper we use $\mathbb{N} = \{1, 2, 3, \ldots \}$ to denote the set of all natural numbers.
Whenever sums are indexed over the empty set we take them to be zero by default.

\section{Preliminaries on Negative Type and Generalized Roundness}
The notions of negative type and generalized roundness --- the formal definitions of which are given in Definition \ref{NTGRDEF} ---
were introduced and studied by Menger \cite{M} and Schoenberg \cite{S1}, \cite{S2}, and Enflo \cite{E2}, respectively.
In part, Schoenberg's studies were focussed on determining which metric spaces can be isometrically
embedded into a Hilbert space. This work was later generalized to the setting of $L_{p}$-spaces ($0 < p \leq 2$)
by Bretagnolle \textit{et al}.\ \cite{BDK} who obtained the following
celebrated characterization: a real (quasi) normed space is linearly isometric to a subspace of some $L_{p}$-space ($0 < p \leq 2$) if and
only if it has $p$-negative type. There are also results along these lines which deal with the less tractable (commutative) case
$p>2$, and with certain of the non commutative $L_{p}$-spaces. See, for example, the papers of Koldobsky and K\"{o}nig \cite{KK}, and Junge \cite{J},
respectively. General references on the interplay between $p$-negative type inequalities and isometric embeddings include Deza and Laurent \cite{DL},
and Wells and Williams \cite{WW}.

Enflo \cite{E2} was interested in a problem of Smirnov concerning uniform embeddings of metric spaces
into Hilbert spaces. A uniform embedding of one metric space into another is a uniformly continuous injection
whose inverse is also uniformly continuous. In other words, uniform embeddings are uniform homeomorphisms
onto their range. Smirnov had asked is every separable metric space uniformly homeomorphic to a subset
of a Hilbert space? In other words; is $L_{2}[0,1]$ a universal uniform embedding space? Enflo answered
Smirnov's question negatively by proving that universal uniform embedding spaces cannot have generalized
roundness-$p$ for any $p > 0$, and by showing that all Hilbert spaces necessarily have generalized roundness-$2$.
In fact, it follows from Enflo's proof that the Banach space of null sequences $c_{0}$ does not embed uniformly into any Hilbert space.
The ideas and constructions in Enflo \cite{E2} have proven extremely useful over time, not only within mainstream functional
analysis, but also in other important areas such as coarse geometry.
The recent monograph of Benyamini and Lindenstrauss \cite{BL}
gives an extensive account of the non linear classification of Banach spaces, including a chapter on uniform
embeddings into Hilbert spaces.

Instigating a theory that has turned out to have a number of uncanny parallels with that of uniform embeddings, Gromov \cite{G}
introduced the notion of coarse embeddings of metric spaces. Gromov \cite{G} asked if every separable metric space coarsely embeds
into a Hilbert space? Dranishnikov \textit{et al}.\ \cite{DGLY} gave a negative answer to Gromov's question by using ideas from
Enflo \cite{E2}. Yu \cite{Y} showed that every discrete metric space which
coarsely embeds into a Hilbert space satisfies the Coarse Baum-Connes Conjecture. Using the work of Dranishnikov \textit{et al}.\ \cite{DGLY}
as a starting point, Nowak \cite{N} developed a number of theoretical similarities between coarse embeddings and uniform embeddings.
Connections between generalized roundness, coarse embeddings and certain forms of the Baum-Connes Conjecture have also been
obtained by Lafont and Prassidis \cite{LP}.
Given the prominence of the Coarse Baum-Connes Conjecture to topologists (and to mathematicians in general), and the striking result
of Yu \cite{Y} (above), it is not surprising that a large number of papers have now been written on coarse embeddings. Unfortunately, many of
these papers use the term ``uniform embedding'' when they are really referring to coarse embeddings.
In addition to the source references mentioned above, the final chapters of the monograph of Roe \cite{R} provide
a good overview of recent work on asymptotic dimension and coarse embeddings into Hilbert spaces.

\begin{defn}\label{NTGRDEF} Let $p \geq 0$ and let $(X,d)$ be a metric space. Then:
\begin{enumerate}
\item[(a)] $(X,d)$ has $p$-{\textit{negative type}} if and only if for all natural numbers $n \geq 2$,
all finite subsets $\{x_{1}, \ldots , x_{n} \} \subseteq X$, and all choices of real numbers $\eta_{1},
\ldots, \eta_{n}$ with $\eta_{1} + \cdots + \eta_{n} = 0$, we have:

\begin{eqnarray}\label{NT}
\sum\limits_{1 \leq i,j \leq n} d(x_{i},x_{j})^{p} \eta_{i} \eta_{j} & \leq & 0.
\end{eqnarray}

\item[(b)] $(X,d)$ has \textit{strict} $p$-{\textit{negative type}} if and only if it has $p$-negative type
and the inequality in (a) is strict whenever the scalar $n$-tuple $(\eta_{1}, \ldots , \eta_{n}) \not= \vec{0}$.

\item[(c)] $(X,d)$ has \textit{generalized roundness}-$p$ if and only if for all natural numbers $n \in \mathbb{N}$,
and all choices of points $a_{1}, \ldots , a_{n}, b_{1}, \ldots , b_{n} \in X$, we have:

\begin{eqnarray}\label{GR}
\sum\limits_{1 \leq k < l \leq n} \left\{ d(a_{k},a_{l})^{p} + d(b_{k},b_{l})^{p} \right\} & \leq &
\sum\limits_{1 \leq j,i \leq n} d(a_{j},b_{i})^{p}.
\end{eqnarray}
\end{enumerate}
\end{defn}

\begin{rem}\label{newrem}
In making Definition 2.1 (c) it is important to point out that repetitions among the $a$'s and $b$'s are allowed.
Indeed, allowing repetitions is essential.
We may, however, when making Definition \ref{NTGRDEF} (c), assume that $a_{j} \not= b_{i}$ for all
$i,j$ ($1 \leq i,j \leq n$). This is due to an elementary cancellation of like terms phenomenon that was
first observed by Andrew Tonge (unpublished).
\end{rem}

Notice that if one restricts to $n=2$ in Definition 2.1 (c) then one gets the condition that Enflo \cite{E1} called
\textit{roundness}-$p$. Roundness-$p$ can be viewed as a direct precursor to the linear Banach space notion known as
Rademacher type. Since being distilled in the 1970s, the related notions of type, cotype and $K$-convexity have played
a very prominent r\^{o}le in the development of linear Banach space theory. See, for example, the survey paper of
Maurey \cite{BM}.
There are also non linear or metric notions of type and cotype due to Bourgain \textit{et al}.\ \cite{BMW}, and
Mendel and Naor \cite{MN}, respectively. In particular, Mendel and Naor \cite{MN} apply
metric cotype to completely settle the problem of classifying when $L_{p}$ embeds coarsely or uniformly into $L_{q}$.
There are also connections --- such as Theorem 2.3 in Lennard \textit{et al}.\ \cite{LTW2} --- between generalized roundness
and linear cotype. A number of open problems persist in this direction.

We should point out that our Definition 2.1 (c) is a cosmetic alteration of the
original definition given in Enflo \cite{E2}.
Enflo actually considered the supremum of all $p$'s that satisfy Definition 2.1 (c).
A result of Linial and Naor, which appears in the paper of Naor and Schechtman \cite{NS},
says that every metric tree has (maximal) roundness two.
The results of Section $5$ of this paper, which develop lower bounds on the $p$-negative type of
finite metric trees, can be thought of as a natural extension of the work of Linial and Naor.

Papers by Lennard \textit{et al}.\ \cite{LTW} and Weston \cite{W} have shown that Definitions 2.1 (a),
2.1 (c) and a third, closely related, condition are all equivalent.
These equivalences are given in Theorem \ref{GRUNT} (below) and, as they are quite central to the rest
of this paper, we include a brief proof for easy reference.
The following definition will help us to state the third condition of Theorem \ref{GRUNT}
succinctly and will moreover be important in its own right throughout the entire paper.

\begin{defn}\label{CUTIE}
Let $X$ be a set. Let $q,t$ be natural numbers.
\begin{enumerate}
\item[(a)] A $(q,t)$-\textit{simplex} in $X$ is a $(q+t)$-vector $(a_{1}, \ldots , a_{q}, b_{1}, \ldots ,b_{t}) \in X^{q+t}$
whose coordinates consist of $q+t$ distinct vertices $a_{1}, \ldots, a_{q}, b_{1}, \ldots , b_{t} \in X$. Such a
simplex will be denoted $D=[a_{j};b_{i}]_{q,t}$.

A vertex $x \in D$ is said to be of \textit{simplex parity} $a$ if $x=a_{j}$ for some $j$, $1 \leq j \leq q$. A vertex
$y \in D$ is said to be of \textit{simplex parity} $b$ if $y=b_{i}$ for some $i$, $1 \leq i \leq t$. Two distinct
vertices $x,y \in D$ are said to be of the \textit{same simplex parity} if they both have simplex parity $a$ or if they
both have simplex parity $b$. And, \textit{opposite simple parity} has the obvious meaning.

\item[(b)] A \textit{load vector} for a $(q,t)$-simplex $D=[a_{j};b_{i}]_{q,t}$ in $X$ is an arbitrary vector
$\vec{\omega} = (m_{1}, \ldots m_{q}, n_{1}, \ldots , n_{t}) \in \mathbb{R}^{q+t}_{+}$ that assigns a positive weight
$m_{j} > 0$ or $n_{i} > 0$ to each vertex $a_{j}$ or $b_{i}$ of $D$ ($1 \leq j \leq q$, $1 \leq i \leq t$), respectively.

\item[(c)] A \textit{loaded} $(q,t)$-\textit{simplex} in $X$ consists of a $(q,t)$-simplex $D=[a_{j};b_{i}]_{q,t}$ in $X$
together with a load vector $\vec{\omega} =(m_{1}, \ldots ,m_{q}, n_{1}, \ldots, n_{t})$ for $D$. Such a loaded simplex
will be denoted $D(\vec{\omega})$ or $[a_{j}(m_{j});b_{i}(n_{i})]_{q,t}$ as the need arises.

\item[(d)] A \textit{normalized} $(q,t)$-\textit{simplex} in $X$ is a loaded $(q,t)$-simplex $D(\vec{\omega})$ in $X$ whose
load vector $\vec{\omega}=(m_{1}, \ldots , m_{q}, n_{1}, \ldots , n_{t})$ satisfies the two normalizations:
\[
m_{1} + \cdots + m_{q} = 1 = n_{1} + \cdots n_{t}.
\]
Such a vector $\vec{\omega}$ will be called a \textit{normalized load vector} for $D$.
\end{enumerate}
\end{defn}

\begin{thm}[Lennard \textit{et al}.\ \cite{LTW}, Weston \cite{W}]\label{GRUNT}
Let $p \geq 0$. For a metric space $(X,d)$, the following conditions are equivalent:
\begin{enumerate}
\item[(a)] $(X,d)$ has $p$-negative type.

\item[(b)] $(X,d)$ has generalized roundness-$p$.

\item[(c)] For all $q,t \in \mathbb{N}$ and all normalized $(q,t)$-simplexes $D(\vec{\omega}) = [a_{j}(m_{j});b_{i}(n_{i})]_{q,t}$
in $X$ we have:
\begin{eqnarray}\label{TWO}
\sum\limits_{1 \leq j_{1} < j_{2} \leq q} m_{j_{1}}m_{j_{2}}d(a_{j_{1}},a_{j_{2}})^{p} +
\sum\limits_{1 \leq i_{1} < i_{2} \leq t} n_{i_{1}}n_{i_{2}}d(b_{i_{1}},b_{i_{2}})^{p} & ~ & \\
\leq \sum\limits_{j,i=1}^{q,t} m_{j}n_{i}d(a_{j},b_{i})^{p}. \hspace*{2.11in} & ~ & ~ \nonumber
\end{eqnarray}
\end{enumerate}
\end{thm}

\begin{proof}[Sketch] The equivalence of conditions (a) and (c) is an easy consequence of the following
observation. Suppose $n \geq 2$ is a natural number. Let $\{ x_{1}, \ldots, x_{n} \} \subseteq X$,
and real numbers $\eta_{1}, \ldots, \eta_{n}$ (not all zero) such that $\eta_{1} + \cdots + \eta_{n}=0$,
be given. By relabeling (if necessary) we may assume there exist $q,t \in \mathbb{N}$ such that
$q+t=n$, $\eta_{1}, \ldots, \eta_{q} \geq 0$, and $\eta_{q+1}, \ldots, \eta_{q+t} < 0$.
Clearly $\sum_{j=1}^{q} \eta_{j} = - \sum_{k=q+1}^{n} \eta_{k}$.
We now make the following designations: For $1 \leq j \leq q$, set $a_{j} = x_{j}$ and $m_{j} = \eta_{j}$.
Further, if $j > q$, we nominally set $m_{j}=0$. For $1 \leq i \leq t$, set $b_{i}=x_{n-i+1}$ and
$n_{i} = - \eta_{n-i+1}$. Further, if $i > t$, we nominally set $n_{i} =0$. For all $k$, $1 \leq k \leq n$,
we then have $\eta_{k} = m_{k} - n_{k}$. More importantly, for any $p \geq 0$, we observe that:
\begin{eqnarray}\label{n5}
\sum\limits_{1 \leq i,j \leq n} d(x_{i},x_{j})^{p} \eta_{i} \eta_{j} & = &
\sum\limits_{1 \leq i,j \leq n} d(x_{i},x_{j})^{p} (m_{i} - n_{i})(m_{j} - n_{j}) \nonumber \\
& = & \sum\limits_{1 \leq j_{1},j_{2} \leq q} m_{j_{1}}m_{j_{2}} d(a_{j_{i}}, a_{j_{2}})^{p}
+ \sum\limits_{1 \leq i_{1},i_{2} \leq t} n_{i_{1}}n_{i_{2}} d(b_{i_{1}}, b_{i_{2}})^{p} \nonumber \\
& ~ & -2 \sum\limits_{j,i = 1}^{n} m_{j}n_{i} d(a_{j},b_{i})^{p}.
\end{eqnarray}
Clearly weights ($\eta_{k}, m_{j}$ or $n_{i}$) that are equal to zero,
and the vertices to which they correspond, play no r\^{o}le in the
determination of (\ref{n5}). Moreover, we may assume that
$\sum_{j=1}^{q} m_{j} = 1 = \sum_{i=1}^{t} n_{i}$ by a simple normalization. Further, the entire
process is clearly symmetric. One may instead start with a normalized $(q,t)$-simplex and simply
\textit{reverse} all of the above designations. The equivalence of conditions (a) and (c) is now plain.

Finally, condition (c) obviously implies condition (b). The converse follows from Remark \ref{newrem}
and a simple density/continuity argument.
\end{proof}

\begin{rem}\label{REMGR}
One advantage of working with condition (c) in Theorem \ref{GRUNT} is that it automatically excludes
the trivial cases of equality that are allowed to occur in the inequalities of conditions (a) and (b).
Hence Theorem \ref{GRUNT} (c) provides an alternate characterization of strict $p$-negative type
when $p > 0$. Namely:
A metric space $(X,d)$ has strict $p$-negative type if and only if the inequality (\ref{TWO})
is strict for each normalized $(q,t)$-simplex $D(\vec{\omega})=[a_{j}(m_{j});b_{i}(n_{i})]_{q,t}$ in $X$.
(This statement is a new theorem in its own right.
The proof is immediate from the equality (\ref{n5}) derived in the proof of Theorem \ref{GRUNT}.)
We will use this result frequently and with little further comment.
\end{rem}

Motivated by the above inequalities (\ref{TWO}) in the particular case $p=1$, we now introduce
two parameters $\gamma_{D}(\vec{\omega})$
and $\Gamma_{X}$ that are designed to ``quantify the degree of strictness'' of the (strict) $1$-negative type inequalities.
We will see that these ``gap'' parameters
are particularly meaningful in the context of finite metric spaces, and especially so for finite metric trees.
The two relevant definitions are as follows.

\begin{defn}\label{SGAP} Let $(X,d)$ be a metric space. Let $q,t$ be natural numbers. Let $D=[a_{j};b_{i}]_{q,t}$ be a
$(q,t)$-simplex in $X$. Let $N_{q,t} \subset \mathbb{R}^{q+t}_{+}$ denote the
set of all normalized load vectors $\vec{\omega}=
(m_{1}, \ldots , m_{q}, n_{1}, \ldots , n_{t})$ for $D$. Then, the $1$-\textit{negative type simplex gap} of $D$ is the
function $\gamma_{D} : N_{q,t} \rightarrow \mathbb{R} : \vec{\omega} \mapsto \gamma_{D}(\vec{\omega})$ where:
\begin{eqnarray*}
\gamma_{D}(\vec{\omega}) & = & \sum\limits_{j,i = 1}^{q,t} m_{j}n_{i}d(a_{j},b_{i}) \\
& ~ & - \sum\limits_{1 \leq j_{1} < j_{2} \leq q} m_{j_{1}}m_{j_{2}}d(a_{j_{1}},a_{j_{2}})
- \sum\limits_{1 \leq i_{1} < i_{2} \leq t} n_{i_{1}}n_{i_{2}}d(b_{i_{1}},b_{i_{2}}),
\end{eqnarray*}
for each $\vec{\omega}=(m_{1}, \ldots, m_{q}, n_{1}, \ldots, n_{t}) \in N_{q,t}$. If we further define the quantities
\begin{eqnarray*}
\mathfrak{R}_{D}(\vec{\omega}) & = & \sum\limits_{j,i =1}^{q,t} m_{j}n_{i}d(a_{j},b_{i}), \, {\rm{and}} \\
\mathfrak{L}_{D}(\vec{\omega}) & = & \sum\limits_{1 \leq j_{1} < j_{2} \leq q} m_{j_{1}}m_{j_{2}}d(a_{j_{1}},a_{j_{2}}) +
\sum\limits_{1 \leq i_{1} < i_{2} \leq t} n_{i_{1}}n_{i_{2}}d(b_{i_{1}},b_{i_{2}}),
\end{eqnarray*}
then we see that $\gamma_{D}(\vec{\omega}) = \mathfrak{R}_{D}(\vec{\omega}) - \mathfrak{L}_{D}(\vec{\omega})$ is the right
hand side of the generalized roundness-$1$
inequality (\ref{TWO}) for the normalized $(q,t)$-simplex $D(\vec{\omega})$ in $X$
subtract the left hand side of the same inequality. So, in particular, $(X,d)$ has strict $1$-negative type if and only
if $\gamma_{D}(\vec{\omega}) > 0$ for each normalized $(q,t)$-simplex $D(\vec{\omega})$ in $X$.
\end{defn}

\begin{defn}\label{NGAP}
Let $(X,d)$ be a metric space with $1$-negative type. We define the $1$-\textit{negative type gap} of $(X,d)$
to be the non negative quantity $$\Gamma_{X} = \inf\limits_{D(\vec{\omega})} \gamma_{D}(\vec{\omega})$$ where
the infimum is taken over all normalized $(q,t)$-simplexes $D(\vec{\omega})$ in $X$.
\end{defn}

Notice that if the $1$-negative type gap $\Gamma_{X} > 0$, then $(X,d)$ has strict $1$-negative type.
Example \ref{STAR2} (given in Section $5$) will show that the converse is not true in general.
In other words, there exist metric spaces $(X,d)$ with strict $1$-negative type and with $\Gamma_{X} = 0$.

\begin{rem}\label{NGAPREM}
More generally, and in the obvious way (again based on (\ref{TWO})), we can define the $p$-\textit{negative
type simplex gap} $\gamma_{D,p} : N_{q,t} \rightarrow \mathbb{R}$ and the resulting $p$-\textit{negative type
gap} $\Gamma_{X,p} = \inf\limits \gamma_{D,p}(\vec{\omega})$ for any metric space $(X,d)$
and any $p \geq 0$. (So that $\gamma_{D} = \gamma_{D,1}$ and $\Gamma_{X} = \Gamma_{X,1}$.) However, for the most
part, our primary interest is the case $p=1$.
\end{rem}

We will see in Section $5$  that if the $1$-negative type gap $\Gamma_{X}$ of a finite metric space $(X,d)$ is positive,
then there must exist a constant $\zeta > 0$ such that $(X,d)$ has strict $p$-negative type for all $p \in (1-\zeta,1+\zeta)$.
This is done in Theorem \ref{LOWERBOUNDS}. The proof of this theorem is independent of the following two sections and the interested reader
may therefore choose to cut ahead and read it now. Moreover, Theorem \ref{LOWERBOUNDS}, which pertains to the case $p=1$,
actually holds for any $p > 0$ provided $\Gamma_{X} > 0$ is replaced by $\Gamma_{X,p} > 0$, and so on. We will return to this
point in Section $5$.

\section{Determining the Simplex Gap of a Finite Metric Tree}
Hjorth \textit{et al}.\ \cite{HLM} have shown that finite metric trees have strict $1$-negative type.
In relation to Definition \ref{NGAP} it therefore makes sense to ask if we can compute the
$1$-negative type gap $\Gamma_{T}$ of an arbitrary finite metric tree $(T,d)$?
The main purpose of these next two sections is to definitively answer this question
\textit{positively}. Our culminating result in this direction is Corollary \ref{GAPFORMULA}.

Our point of entry for the above question will be to develop a key formula for the simplex gap
evaluation $\gamma_{D}(\vec{\omega})$
of a normalized $(q,t)$-simplex $D(\vec{\omega})$ in a finite metric tree $(T,d)$.
This is done in Theorem \ref{FORMULA} and it
will eventually allow the exact computation of the $1$-negative type gap $\Gamma_{T} = \inf \gamma_{D}$ of $(T,d)$.
Prior to doing this, however, it is highly germane to review some basic facts and standard notations
pertaining to finite metric trees. We will also introduce some concepts and notations that are less standard.

\begin{defn}
A finite metric tree is a finite connected graph $T$ that has no cycles,
endowed with an edge weighted path metric $d$.
Terminal vertices in $T$ are called \textit{leaves} or \textit{pendants}.
Given vertices $x,y \in T$, the unique shortest path from $x$ to $y$ is called a \textit{geodesic} and
is denoted $[x,y]$. In particular, the pair $e=(x,y)$ is an edge in $T$ if and only if the
geodesic $[x,y]$ from $x$ to $y$ contains no other vertices of $T$. If an edge $e$ lies on a geodesic $[x,y]$,
we may sometimes write $e \subseteq [x,y]$.
\end{defn}

\begin{notation} Given an edge $e=(x,y)$ in a finite metric tree $(T,d)$ we will often find it convenient
use the notation $|e| = d(x,y)$ to denote the metric length of the edge.
\end{notation}

\begin{defn}
Let $(T,d)$ be a finite metric tree.
\begin{enumerate}
\item[(a)] If $|e|=1$ for all edges $e=(x,y)$ in $(T,d)$ we will say that the path metric $d$ is \textit{ordinary}
or \textit{unweighted}.

\item[(b)] More generally; if $|e| \not= 1$ for at least one edge $e=(x,y)$ in $(T,d)$, we will say that the path
metric $d$ is \textit{edge weighted}.
\end{enumerate}
\end{defn}

\begin{defn}
Given a finite metric tree $(T,d)$ and a set of vertices $V \subseteq T$ we can form the smallest subtree of $T$
that contains all the vertices of $V$ --- denoted by $T_{V}$ --- and we can endow it with the natural restriction
of the metric $d$. We will call $(T_{V},d)$ the \textit{minimal subtree} of $(T,d)$ generated by the set of vertices $V$.
Clearly: if $V=\{v_{1}, \ldots ,v_{k} \} \subseteq T$ then the minimal subtree $T_{V}$ consists of all vertices $x \in T$
that lie on some geodesic $[v_{i},v_{j}]$ in $T$. Of course, the minimal subtree $(T_{V},d)$ is a finite metric tree
in its own right. Given a subset $V \subseteq T$ it is also clear that $T_{V}=T$ if and only if $V$ contains all
the leaves of $T$.
\end{defn}

The following definition introduces a convention to ``orient'' the edges in any given tree.
This will enable the treatment of edges as ordered pairs in a systematic and unambiguous way.
Orientation will play a key r\^{o}le in determining the main results of this paper.

\begin{defn}\label{CROSS}
Let $(T,d)$ be a finite metric tree. By way of convention, we choose and then highlight a fixed leaf $\ell \in T$.
This distinguished leaf $\ell$ is then called the \textit{root} of $T$. Once the root has been
fixed we may make the following definitions.
\begin{enumerate}
\item[(a)] An edge $e=(x,y)$ in $T$ is
(\textit{left/right}) \textit{oriented} if $d(x, \ell) > d(y, \ell)$. In other words, an oriented edge in $T$
is an ordered pair $e=(x,y)$ of adjacent vertices $x,y \in T$ where $x$ is geodesically further from the
root $\ell$ than $y$. The set of all such oriented edges $e$ in $T$ will be denoted $E(T)$.

\item[(b)] A vertex $v \in T$ is to the \textit{left} of an oriented edge $e =(x,y) \in E(T)$ if $d(v,x) < d(v,y)$.
If it is also the case that $v \not= x$ then we will say that $v$ is \textit{strictly to the left} of $e$.
The set of all vertices $v \in T$ that are to the left of $e$ will be denoted $L(e)$. And the set of all
vertices $v \in T$ that are strictly to the left of $e$ will be denoted $\overline{L}(e)$. Notice that we
always have $x \in L(e)$ but it can happen that $\overline{L}(e) = \emptyset$. Alternately, we may think
of $L(e)$ as the vertices of the subtree that is rooted at $x$ (oriented as per $T$).

\item[(c)] A vertex $v \in T$ is to the \textit{right} of an oriented edge $e=(x,y) \in E(T)$ if $d(v,y) < d(v,x)$.
If it is also the case that $v \not= y$ then we will say that $v$ is \textit{strictly to the right of} $e$.
The set of all vertices $v \in T$ that are to the right of $e$ will be denoted $R(e)$. And the set of all
vertices $v \in T$ that are strictly to the right of $e$ will be denoted $\overline{R}(e)$.
\end{enumerate}
Notice that each oriented edge $e \in E(T)$ partitions the vertices of $T$ into a disjoint union $L(e) \cup R(e)$.
\end{defn}

Henceforth, whenever we are referring to a particular finite metric tree, it will be understood that a root
leaf has been chosen from the outset. So ``edges'' are now always ordered pairs $e=(x,y)$ with the left vertex
$x$ as the first coordinate and the right vertex $y$ as the second coordinate. In particular, orientation affords
the following compact notation.

\begin{notation}
Given an oriented edge $e=(x,y)$ in a finite metric tree $(T,d)$ we may use its unique left vertex $x$ to
alternately denote the edge as $e(x)$. Note that, under this scheme, $e(\ell)$ is not defined because the
root leaf $\ell$ is not the left vertex of any oriented edge. All other vertices in $T$ appear
(uniquely) as the left vertex of some oriented edge.
\end{notation}

\begin{defn}\label{PARTSUM}
Let $D=[a_{j};b_{i}]_{q,t}$ be a fixed $(q,t)$-simplex in a finite metric tree $(T,d)$. Let $T_{D}$ be the
minimal subtree of $T$ generated by the vertices $a_{j},b_{i}$ of $D$. Orient the edges of $T_{D}$ by fixing
a root leaf $\ell \in T_{D}$. For each oriented edge $e \in E(T_{D})$ and each load vector
$\vec{\omega} = (m_{1}, \ldots, m_{q}, n_{1}, \ldots , n_{t}) \in \mathbb{R}_{+}^{q+t}$ for $D$, we define
the following \textit{partition sums of} $\vec{\omega}$:
\begin{enumerate}
\item[(a)] $\alpha_{L}(\vec{\omega},e) = \sum\limits_{j \in A_{L}(e)} m_{j}$ where $A_{L}(e) = \{ j \in [q] : a_{j} \in L(e) \}$.

\item[(b)] $\alpha_{R}(\vec{\omega},e) = \sum\limits_{j \in A_{R}(e)} m_{j}$ where $A_{R}(e) = \{ j \in [q] : a_{j} \in R(e) \}$.

\item[(c)] $\beta_{L}(\vec{\omega},e)  = \sum\limits_{i \in B_{L}(e)} n_{i}$ where $B_{L}(e) = \{ i \in [t] : b_{i} \in L(e) \}$.

\item[(d)] $\beta_{R}(\vec{\omega},e)  = \sum\limits_{i \in B_{R}(e)} n_{i}$ where $B_{R}(e) = \{ i \in [t] : b_{i} \in R(e) \}$.
\end{enumerate}

If, in the above definitions, we replace $L(e)$ and $R(e)$ with $\overline{L}(e)$ and $\overline{R}(e)$ (respectively), then we
obtain the \textit{strict partition sums of} $\vec{\omega}$: $\overline{\alpha}_{L}(\vec{\omega},e)$, $\overline{\alpha}_{R}(\vec{\omega},e)$,
$\overline{\beta}_{L}(\vec{\omega},e)$ and $\overline{\beta}_{R}(\vec{\omega},e)$. For example:
\begin{enumerate}
\item[(e)] $\overline{\alpha}_{L}(\vec{\omega},e) = \sum \{ m_{j} : a_{j} \in \overline{L}(e) \}$.

\item[(f)] $\overline{\beta}_{L}(\vec{\omega},e)  = \sum \{ n_{i} : b_{i} \in \overline{L}(e) \}$.
\end{enumerate}

Notice that if the load vector $\vec{\omega}$ is normalized, then we obtain the innocuous looking (but important) identities
$\alpha_{L}(\vec{\omega},e) + \alpha_{R}(\vec{\omega},e) = 1 = \beta_{L}(\vec{\omega},e) + \beta_{R}(\vec{\omega},e)$.
\end{defn}

\begin{notation}
In relation to Definition \ref{PARTSUM}, if we want to emphasize the (fixed) underlying $(q,t)$-simplex $D$, we may sometimes
write $\alpha_{L}(D,\vec{\omega},e)$ in place of $\alpha_{L}(\vec{\omega},e)$, and so on. (See, for example, Lemma \ref{REDUCE}.)
\end{notation}

\begin{thm}\label{FORMULA}
Let $D=[a_{j};b_{i}]_{q,t}$ be a given $(q,t)$-simplex in a finite metric tree $(T,d)$.
Let $T_{D}$ denote the minimal subtree of $T$ generated by the vertices of $D$. Let $N_{q,t} \subset \mathbb{R}_{+}^{q+t}$
denote the set of all normalized load vectors for $D$. Then, for each such normalized load vector
$\vec{\omega} =(m_{1}, \ldots, m_{q}, n_{1}, \ldots, n_{t}) \in N_{q,t}$, the simplex gap evaluation $\gamma_{D}(\vec{\omega})$ is
given by the following formulas:
\begin{eqnarray*}
\gamma_{D}(\vec{\omega}) & = & \sum\limits_{e \in E(T_{D})} (\alpha_{L}(\vec{\omega},e)-\beta_{L}(\vec{\omega},e))^{2} \cdot |e| \\
                         & = & \sum\limits_{e \in E(T_{D})} (\alpha_{R}(\vec{\omega},e)-\beta_{R}(\vec{\omega},e))^{2} \cdot |e|.
\end{eqnarray*}
In particular it follows that the simplex gap functions $\gamma_{D} : N_{q,t} \rightarrow \mathbb{R}$ are positive
valued for all possible $(q,t)$-simplexes $D \subseteq T$.
\end{thm}

\begin{proof} Fix a normalized load vector $\vec{\omega} =(m_{1}, \ldots, m_{q}, n_{1}, \ldots, n_{t})$ for the given
$(q,t)$-simplex $D=[a_{j};b_{i}]_{q,t}$. The idea of the proof is to calculate the contribution of each oriented
edge $e \in E(T_{D})$ to the simplex gap evaluation $\gamma_{D}(\vec{\omega})$, and then to sum over all such oriented edges.

As per Definition \ref{SGAP}, $\gamma_{D}(\vec{\omega}) = \mathfrak{R}_{D}(\vec{\omega}) - \mathfrak{L}_{D}(\vec{\omega})$, where
\begin{eqnarray*}
\mathfrak{L}_{D}(\vec{\omega}) & = & \sum\limits_{1 \leq j_{1} < j_{2} \leq q} m_{j_{1}}m_{j_{2}}d(a_{j_{1}},a_{j_{2}})
+ \sum\limits_{1 \leq i_{1} < i_{2} \leq t} n_{i_{1}}n_{i_{2}}d(b_{i_{1}},b_{i_{2}}), \, {{\rm and}} \\
\mathfrak{R}_{D}(\vec{\omega}) & = & \sum\limits_{j,i = 1}^{q,t} m_{j}n_{i}d(a_{j},b_{i}).
\end{eqnarray*}

Notice that if $[x,y]$ is a geodesic in the minimal subtree $T_{D}$, then:
\begin{eqnarray}\label{GEODESIC}
d(x,y) & = & \sum \bigl\{ |f| : f \in E(T_{D}) \,\text{and}\, f \subseteq [x,y] \bigl\}.
\end{eqnarray}
This is because $(T_{D},d)$ is a metric tree. Due to the geodesic decompositions (\ref{GEODESIC})
we may therefore rewrite the sums $\mathfrak{L}_{D}(\vec{\omega})$ and $\mathfrak{R}_{D}(\vec{\omega})$ as
\[
\mathfrak{L}_{D}(\vec{\omega}) = \sum\limits_{e \in E(T_{D})} \mathfrak{L}_{D}^{(e)}(\vec{\omega})
\cdot |e|, \,\,{\rm and}\,\,
\mathfrak{R}_{D}(\vec{\omega}) = \sum\limits_{e \in E(T_{D})} \mathfrak{R}_{D}^{(e)}(\vec{\omega}) \cdot |e|,
\]
where the coefficients $\mathfrak{L}_{D}^{(e)}(\vec{\omega})$ and
$\mathfrak{R}_{D}^{(e)}(\vec{\omega})$ are yet to be determined.

Now consider a fixed oriented edge $e \in E(T_{D})$. Notice that if the edge $e$
lies on the geodesic $[a_{j_{1}},a_{j_{2}}]$ then
the term $m_{j_{1}}m_{j_{2}} \cdot |e|$ appears in the sum $\mathfrak{L}_{D}(\vec{\omega})$ (and so on). For this to happen,
$a_{j_{1}}$ must be to the left of $e$ (that is, $j_{1} \in A_{L}(e)$) and $a_{j_{2}}$ must be to the right of $e$
(that is, $j_{2} \in A_{R}(e)$) or, \textit{vice versa}. This and similar such comments, together with
the definitions of $\mathfrak{L}_{D}(\vec{\omega})$ and $\mathfrak{R}_{D}(\omega)$, imply:
\begin{eqnarray*}
\mathfrak{L}_{D}^{(e)}(\vec{\omega}) & = &
\left( \sum\limits_{j_{1} \in A_{L}(e) } m_{j_{1}} \right) \left( \sum\limits_{j_{2} \in A_{R}(e) } m_{j_{2}} \right) +
\left( \sum\limits_{i_{1} \in B_{L}(e) } n_{i_{1}} \right) \left( \sum\limits_{i_{2} \in B_{R}(e) } n_{i_{2}} \right) \\
         & = &
\alpha_{L}(\vec{\omega},e) \cdot \alpha_{R}(\vec{\omega},e) + \beta_{L}(\vec{\omega},e) \cdot \beta_{R}(\vec{\omega},e) \\
         & = &
\alpha_{L}(\vec{\omega},e) \cdot (1 - \alpha_{L}(\vec{\omega},e)) + \beta_{L}(\vec{\omega},e) \cdot (1 - \beta_{L}(\vec{\omega},e)), \, {{\rm and}} \\
\mathfrak{R}_{D}^{(e)}(\omega) & = &
\left( \sum\limits_{j \in A_{L}(e) } m_{j} \right) \left( \sum\limits_{i \in B_{R}(e) } n_{i} \right) +
\left( \sum\limits_{j \in A_{R}(e) } m_{j} \right) \left( \sum\limits_{i \in B_{L}(e) } n_{i} \right) \\
         & = &
\alpha_{L}(\vec{\omega},e) \cdot \beta_{R}(\vec{\omega},e) + \alpha_{R}(\vec{\omega},e) \cdot \beta_{L}(\vec{\omega},e) \\
         & = &
\alpha_{L}(\vec{\omega},e) \cdot (1 - \beta_{L}(\vec{\omega},e)) + (1 - \alpha_{L}(\vec{\omega},e)) \cdot \beta_{L}(\vec{\omega},e).
\end{eqnarray*}

We can now define $\gamma_{D}^{(e)}(\vec{\omega})$, the contribution of the oriented edge $e \in E(T_{D})$
to the simplex gap evaluation $\gamma_{D}(\vec{\omega})$, in a natural and obvious way:
$$\gamma_{D}^{(e)}(\vec{\omega}) = \left(\mathfrak{R}_{D}^{(e)}(\vec{\omega})
- \mathfrak{L}_{D}^{(e)}(\vec{\omega})\right) \cdot |e|.$$
As a result we get the following simplex gap decomposition automatically:
\[
\gamma_{D}(\vec{\omega}) = \mathfrak{R}_{D}(\vec{\omega}) - \mathfrak{L}_{D}(\vec{\omega}) = \sum\limits_{e \in E(T_{D})} \gamma_{D}^{(e)}(\vec{\omega}).
\]

Setting $\alpha = \alpha_{L}(\vec{\omega},e)$ and $\beta = \beta_{L}(\vec{\omega},e)$ we see, from the preceding computations, that:
\begin{eqnarray*}
\gamma_{D}^{(e)}(\vec{\omega}) & = & \left(\mathfrak{R}_{D}^{(e)}(\vec{\omega}) - \mathfrak{L}_{D}^{(e)}(\vec{\omega})\right) \cdot |e| \\
                               & = & (\alpha \cdot (1 - \beta) + (1 - \alpha) \cdot \beta - \alpha \cdot (1 - \alpha) -
                                     \beta \cdot (1 - \beta)) \cdot |e| \\
                               & = & (\alpha^{2} - 2 \alpha \beta + \beta^{2}) \cdot |e| \\
                               & = & (\alpha - \beta)^{2} \cdot |e|\\
                               & = & (\alpha_{L}(\vec{\omega},e) - \beta_{L}(\vec{\omega},e))^{2} \cdot |e| \\
                               & = & (\alpha_{R}(\vec{\omega},e) - \beta_{R}(\vec{\omega},e))^{2} \cdot |e|.
\end{eqnarray*}
Now sum $\gamma_{D}^{(e)}(\vec{\omega})$ over all $e \in E(T_{D})$ to get the stated formulas for $\gamma_{D}(\vec{\omega})$.

If either vertex of an oriented edge $e$ is a leaf in the minimal subtree $T_{D}$, then clearly $\gamma_{D}^{(e)}(\vec{\omega}) > 0$
and hence the simplex gap $\gamma_{D}(\vec{\omega}) > 0$, establishing the final statement of the theorem.
\end{proof}

\begin{notation}
As introduced in the proof of Theorem \ref{FORMULA}, given a normalized $(q,t)$-simplex $D(\vec{\omega})$ in a finite metric tree
$(T,d)$, we will continue to use the notation $\gamma_{D}^{(e)}(\vec{\omega})$ to denote the contribution of an oriented edge $e \in E(T_{D})$
to the simplex gap evaluation $\gamma_{D}(\vec{\omega})$. So, according to Theorem \ref{FORMULA}, we
have the following formulas:
\begin{enumerate}
\item[(a)] $\gamma_{D}^{(e)}(\vec{\omega}) = \left( \alpha_{L}(\vec{\omega},e) - \beta_{L}(\vec{\omega},e) \right)^{2} \cdot |e|
= \left( \alpha_{R}(\vec{\omega},e) - \beta_{R}(\vec{\omega},e) \right)^{2} \cdot |e|$ for each oriented edge $e \in E(T_{D})$, and

\item[(b)] $\gamma_{D}(\vec{\omega}) = \sum\limits_{e \in E(T_{D})} \gamma_{D}^{(e)}(\vec{\omega})$.
\end{enumerate}
\end{notation}

Remark \ref{REMGR}, Definition \ref{SGAP} and Theorem \ref{FORMULA} automatically furnish a new and elementary proof of the
following result of Hjorth \textit{et al}.\ \cite{HLM}.

\begin{cor}\label{HJORTH}
Every finite metric tree has strict $1$-negative type.
\end{cor}

In addition to finite metric trees, Hjorth \textit{et al}.\ \cite{HKM} and Hjorth \textit{et al}.\ \cite{HLM} have
elaborated and studied several other classes of finite metric spaces which have strict $1$-negative type. These
include --- under appropriate restrictions --- finite metric spaces whose elements have been chosen from a Riemannian
manifold (and endowed with the natural inherited distances).

\section{Determining the Negative Type Gap of a Finite Metric Tree}
In this section we compute the exact value of the $1$-negative type gap $\Gamma_{T}$ (see Definition \ref{NGAP})
of a finite metric tree $(T,d)$, and then explore some consequences of this computation. We begin with an upper bound
\[
\Biggl\{ \sum\limits_{e \in E(T)} |e|^{-1} \Biggl\}^{-1}
\]
for $\Gamma_{T}$ that is determined via an algorithm, and then proceed to show that
this upper bound is also a lower
bound for $\Gamma_{T}$. Isolating the value of $\Gamma_{T}$ leads to an entirely new class of inequalities for
finite metric trees which may be termed \textit{inequalities of enhanced $1$-negative type}. These inequalities
are developed in Theorems \ref{MAIN} and \ref{REFORM}. Not surprisingly, we need to introduce some more
definitions and concepts before computing $\Gamma_{T}$. These are as follows.

\begin{defn}\label{LEVELS}
Let $T$ be a finite tree. Let $\ell \in T$ be the designated root leaf for $T$. Let $d$ denote
the ordinary path metric on $T$ and set:
\begin{eqnarray*}
k_{0} = \max\limits_{x \in T} d(x,\ell).
\end{eqnarray*}
Let $k$ be any integer such that $0 \leq k \leq k_{0}$. Then we say that a vertex $v \in T$ is a
\textit{level $k$ vertex of $T$} if $d(v,\ell) = k_{0}-k$.

The introduction of levels has the effect of partitioning $T$ into $k_{0}$ disjoint sets of vertices.
\end{defn}

We will now focus on a particular subclass of normalized $(q,t)$-simplexes $D(\vec{\omega})$ that
turn out to be pivotal in the determination of the $1$-negative type gap $\Gamma_{T}$ of a finite
metric tree $(T,d)$. The condition we introduce depends only upon the underlying tree $T$ and the
vertices of $D$. The path metric $d$ on $T$ and the normalized load vectors $\vec{\omega}$ for $D$
play no (immediate) r\^{o}le. The relevant definition is as follows.

\begin{defn}\label{GENERIC}
Let $T$ be a finite tree. Let $D$ be a $(q,t)$-simplex in $T$. Let $T_{D}$ be the minimal subtree of $T$
generated by the vertices of $D$. We say that $D$ is \textit{generically labeled} if:
\begin{enumerate}
\item[(a)] $D=T_{D}$ as sets (in other words, every vertex of $T_{D}$ belongs to $D$), and

\item[(b)] for all edges $e=(x,y) \in E(T_{D})$, $x$ and $y$ have opposite simplex parity.
\end{enumerate}

Notice that we can restate condition (b) in terms of levels:
\begin{enumerate}
\item[(c)] For all vertices $x,y \in T_{D}$, if $x$ is in an even level of $T_{D}$ and if $y$ is in an
odd level of $T_{D}$, then $x$ and $y$ have opposite simplex parity.
\end{enumerate}
\end{defn}

\begin{rem}\label{TGENERIC}
Let $T$ be a finite tree. Suppose $T$ has been oriented via the designation of a root leaf $\ell \in T$.
We can always generically label the vertices of $T$. The easiest way to describe this process is to use levels.
Simply assign parity $a$ to all vertices of $T$ that lie in even numbered levels, and parity $b$ to all vertices
of $T$ that lie in odd numbered levels. This realizes the whole tree $T$ as a generically labeled $(q,t)$-simplex
with:
\begin{enumerate}
\item[(a)] $q=|\left\{ x \in T: x \right.$ is in an even numbered level of $\left. T \right\}|$, and \vspace*{0.01in}

\item[(b)] $t=|\left\{ y \in T: y \right.$ is in an odd numbered level of $\left. T \right\}|$.
\end{enumerate}
Clearly there is (essentially) only one way to generically label the vertices of $T$. The only other possible labeling
of the vertices of $T$ that is generic is the trivial one whereby we switch all of the parity assignments given above:
$a_{j} \leftrightarrow b_{i}$. We may therefore refer to \textit{the} generic labeling of the vertices of $T$.

In short, generic labeling a finite tree $T$ amounts to little more than a $2$-coloring of the vertices of $T$.
\end{rem}

\begin{defn}\label{GENALG}
Let $T$ be a finite tree.
Let $\ell \in T$ be the designated root leaf of $T$.
Partition the vertices of $T$ into $k_{0}+1$ levels as per Definition \ref{LEVELS}.
Let $D=[a_{j};b_{i}]_{q,t}$ denote the (essentially) unique $(q,t)$-simplex in T
that generically labels the vertices of $T$
as per Remark \ref{TGENERIC}.
We may assume that the level $0$ vertices in $T$ have parity $a$ in the simplex $D$.

Let $k_{1}$ be an arbitrary odd natural number such that $k_{1} \leq k_{0}$, and let
$k_{2}$ be an arbitrary even natural number such that $k_{2} \leq k_{0}$. We may denote the level $k_{1}$ vertices of $T$ as
$b_{i}^{(k_{1})}$ where $i$ ranges over a suitable segment of the natural numbers, and we may denote the level $k_{2}$
vertices of $T$ as $a_{j}^{(k_{2})}$ where $j$ ranges over a suitable segment of the natural numbers. This notation allows
us to ``rewrite'' the generically labeled simplex $D=T$ in the form $D=[a_{j}^{(k_{2})};b_{i}^{(k_{1})}]_{q,t}$.

Now, given $\delta > 0$, we define the following \textit{generic algorithm} that assigns a unique
vector $\vec{\omega}_{\delta} =(m_{j}^{(k_{2})},n_{i}^{(k_{1})}) \in \mathbb{R}^{q+t}$ to the generically labeled simplex $D=T$:
\begin{enumerate}
\item[(a)] Set each level $0$ weight to be:
\begin{eqnarray*}
m_{j}^{(0)} & = & \frac{\delta}{\left|e\left(a_{j}^{(0)}\right)\right|}.
\end{eqnarray*}

\item[(b)] If $k_{1} < k_{0}$ is odd and if weights have been assigned by the algorithm to all level $k$ vertices
of $T$ for all $k < k_{1}$, then set:
\begin{eqnarray*}
n_{i}^{(k_{1})} & = & \alpha_{L}\left(e\left(b_{i}^{(k_{1})}\right)\right) - \overline{\beta}_{L}\left(e\left(b_{i}^{(k_{1})}\right)\right)
+ \frac{\delta}{\left|e\left(b_{i}^{(k_{1})}\right)\right|}
\end{eqnarray*}
for each value of the subscript $i$.

\item[(c)] If $k_{2} < k_{0}$ is even and if weights have been assigned by the algorithm to all level $k$ vertices
of $T$ for all $k < k_{2}$, then set:
\begin{eqnarray*}
m_{j}^{(k_{2})} & = & \beta_{L}\left(e\left( a_{j}^{(k_{2})}\right)\right) - \overline{\alpha}_{L}\left(e\left( a_{j}^{(k_{2})}\right)\right)
+ \frac{\delta}{\left|e\left( a_{j}^{(k_{2})}\right)\right|}
\end{eqnarray*}
for each value of the subscript $j$.

\item[(d)] If $k_{0}$ is odd, then set:
\begin{eqnarray*}
n_{1}^{(k_{0})} & = & 1 - \sum\limits_{\stackrel{\textstyle{i,k}}{k < k_{0}}} n_{i}^{(k)}.
\end{eqnarray*}

\item[(e)] If $k_{0}$ is even, then set:
\begin{eqnarray*}
m_{1}^{(k_{0})} & = & 1 - \sum\limits_{\stackrel{\textstyle{j,k}}{k < k_{0}}} m_{j}^{(k)}.
\end{eqnarray*}
\end{enumerate}
\end{defn}

\begin{lem}\label{GALEM}
Let $T$ be a finite tree. Let $\ell \in T$ be the designated root leaf of $T$. Let $e_{\ell}=(z,\ell)$ denote
the unique oriented edge in $E(T)$ whose right vertex is $\ell$. Let $D=[a_{j}^{(k_{2})};b_{i}^{(k_{1})}]_{q,t}$ denote
the (essentially) unique $(q,t)$-simplex in $T$ that generically labels the vertices of $T$. (Here, as in Definition \ref{GENALG},
superscripts are being used to denote the level of each vertex in the $(q,t)$-simplex $D=T$.) For each $\delta > 0$, let
$\vec{\omega}_{\delta} = (m_{j}^{(k_{2})},n_{i}^{(k_{1})}) \in \mathbb{R}^{q+t}$ be the vector assigned to the
$(q,t)$-simplex $D$ by the generic algorithm. Then:
\begin{enumerate}
\item[(a)] $\vec{\omega}_{\delta}$ is a load vector for the $(q,t)$-simplex $D$ if and only if
\begin{eqnarray*}
\delta & < & \Biggl\{ \sum\limits_{e \in E(T) \setminus \{e_{\ell}\}} |e|^{-1} \Biggl\}^{-1}.
\end{eqnarray*}

\item[(b)] $\vec{\omega}_{\delta}$ is a normalized load vector for the $(q,t)$-simplex $D$ if and only if
\begin{eqnarray*}
\delta & = & \Biggl\{ \sum\limits_{e \in E(T)} |e|^{-1} \Biggl\}^{-1}.
\end{eqnarray*}
\end{enumerate}
\end{lem}

\begin{proof}
Let $\delta > 0$ be given. For simplicity, and using the notations of Definitions \ref{LEVELS} and \ref{GENALG}, we will
assume that $k_{0}$ is even. (The case where $k_{0}$ is odd is entirely similar and is omitted.)

According to the definition of the generic algorithm, only
\begin{eqnarray*}
m_{1}^{(k_{0})} & = & 1 - \sum\limits_{\stackrel{\textstyle{j,k}}{k < k_{0}}} m_{j}^{(k)}
\end{eqnarray*}
is possibly non positive. So $\vec{\omega}_{\delta}$ is a load vector for $D$ if and only if $m_{1}^{(k_{0})} > 0$.
With this in mind, we shall address both parts of the lemma simultaneously.
Applying the definition of the generic algorithm repeatedly leads to the following observations.

If we sum $m_{j}^{(0)}$ for all level zero vertices in $T$ we obtain:
\begin{eqnarray*}
\sum\limits_{j} m_{j}^{(0)} & = & \delta \cdot \sum\limits_{j} \left| e\left( a_{j}^{(0)}\right ) \right|^{-1}.
\end{eqnarray*}

If we sum $n_{i}^{(1)}$ for all level one vertices in $T$ we obtain:
\begin{eqnarray*}
\sum\limits_{i} n_{i}^{(1)} & = & \sum\limits_{j} m_{j}^{(0)}
+ \delta \cdot \sum\limits_{i} \left| e\left( b_{i}^{(1)}\right) \right|^{-1} \\
& = & \delta \cdot \sum\limits_{i} \left| e\left( b_{i}^{(1)}\right)\right|^{-1}
+ \delta \cdot \sum\limits_{j} \left| e\left( a_{j}^{(0)}\right)\right|^{-1},
\end{eqnarray*}
where the last line follows by the previous computation.

If we sum $m_{j}^{(2)}$ for all level two vertices we obtain:
\begin{eqnarray*}
\sum\limits_{j} m_{j}^{(2)} & = &
\biggl\{ \sum\limits_{i} n_{i}^{(1)}  - \sum\limits_{j} m_{j}^{(0)}\biggl\}
+ \, \delta \cdot \sum\limits_{j} \left|e \left( a_{j}^{(2)}\right)\right|^{-1} \\
& = & \delta \cdot \sum\limits_{i} \left| e \left( b_{i}^{(1)} \right) \right|^{-1}
+ \delta \cdot \sum\limits_{j} \left|e \left( a_{j}^{(2)}\right) \right|^{-1},
\end{eqnarray*}
where the last line follows by the previous computation.

We therefore obtain the following recursive formulas by induction:
\begin{eqnarray*}
\sum\limits_{i} n_{i}^{(k_{1})} & = & \delta \cdot \sum\limits_{i} \left| e \left( b_{i}^{(k_{1})} \right) \right|^{-1}
+ \delta \cdot \sum\limits_{j} \left| e \left( a_{j}^{(k_{1}-1)} \right) \right|^{-1}
\end{eqnarray*}
for all odd natural numbers $k_{1}$ such that $k_{1} < k_{0}$, and
\begin{eqnarray*}
\sum\limits_{j} m_{j}^{(k_{2})} & = & \delta \cdot \sum\limits_{j} \left| e \left( a_{j}^{(k_{2})} \right) \right|^{-1}
+ \delta \cdot \sum\limits_{i} \left| e \left( b_{i}^{(k_{2}-1)}\right) \right|^{-1}
\end{eqnarray*}
for all even natural numbers $k_{2}$ such that $0 < k_{2} < k_{0}$. Hence:
\[
x = \sum\limits_{\stackrel{\textstyle{j,k}}{k < k_{0}}} m_{j}^{(k)} = \delta \cdot
\Biggl\{\sum\limits_{e \in E(T) \setminus \{ e_{\ell} \}} |e|^{-1}\Biggl\}.
\]
And so $1-x>0$ if and only if $\delta < \Biggl\{ \sum\limits_{e \in E(T) \setminus \{ e_{\ell}\}} |e|^{-1} \Biggl\}^{-1}$,
establishing Part (a).

Moreover, our recursive formulas show that:
\begin{eqnarray*}
\sum\limits_{i,k} n_{i}^{(k)} & = & \delta \cdot \Biggl\{ \sum\limits_{e \in E(T)} |e|^{-1} \Biggl\}.
\end{eqnarray*}
Therefore $\vec{\omega}_{\delta}$ is normalized if and only if $\delta = \Biggl\{ \sum\limits_{e \in E(T)} |e|^{-1} \Biggl\}^{-1}$.
\end{proof}

\begin{thm}\label{UPPERBOUND}
Let $(T,d)$ be a finite metric tree. Let $\Gamma_{T}$ denote the $1$-negative type gap of $(T,d)$. Then:
\begin{eqnarray*}
\Gamma_{T} & \leq & \Biggl\{ \sum\limits_{e \in E(T)} |e|^{-1} \Biggl\}^{-1}.
\end{eqnarray*}
\end{thm}

\begin{proof}
For any given normalized $(q,t)$-simplex $D(\vec{\omega})$ in $T$, the simplex gap evaluation $\gamma_{D}(\vec{\omega})$ provides an
upper bound for $\Gamma_{T}$ (by definition).

Let $D$ denote the (essentially) unique $(q,t)$-simplex in $T$ that generically labels the vertices of $T$.
Let $\vec{\omega}_{G}$ denote the unique normalized load vector for $D$ that is generated by the generic algorithm.
By Lemma \ref{GALEM}, $\vec{\omega}_{G} = \vec{\omega}_{\delta}$ where:
\begin{eqnarray*}
\delta & = & \Biggl\{ \sum\limits_{e \in E(T)} |e|^{-1}\Biggl\}^{-1}.
\end{eqnarray*}

Consider the resulting normalized $(q,t)$-simplex $D(\vec{\omega}_{G})$ in $T$. The generic algorithm is structured
so that
\begin{eqnarray*}
\bigl|\alpha_{L}(\vec{\omega}_{G},e) - \beta_{L}(\vec{\omega}_{G},e)\bigl| & = & \frac{\delta}{|e|}
\end{eqnarray*}
for all oriented edges $e \in E(T)$.

Hence by Theorem \ref{FORMULA}:
\begin{eqnarray*}
\gamma_{D}(\vec{\omega}_{G}) & = & \sum\limits_{e \in E(T)} \frac{\delta^{2}}{|e|^{2}} \cdot |e| \\
& = & \delta^{2} \cdot \sum\limits_{e \in E(T)} |e|^{-1} \\
& = & \delta^{2} \cdot \delta^{-1} \\
& = & \Biggl\{ \sum\limits_{e \in E(T)} |e|^{-1} \Biggl\}^{-1}.
\end{eqnarray*}
\end{proof}

Theorem \ref{UPPERBOUND} already gives an indication that the generic algorithm is going to be very important
in the context of this paper. We therefore isolate the following natural definition.

\begin{defn}
To say that a finite metric tree $(T,d)$ is \textit{generically labeled and generically weighted} means we are considering the
(essentially) unique normalized $(q,t)$-simplex $D(\vec{\omega}_{G})$ in $T$ with the following properties:
\begin{itemize}
\item[(a)] $q+t = |T|$,

\item[(b)] $D$ is generically labeled, and

\item[(c)] $\vec{\omega}_{G}$ is the unique normalized load vector for $D$ that is generated by the
generic algorithm.
\end{itemize}
\end{defn}

Suppose $D=[a_{j};b_{i}]_{q,t}$ is a $(q,t)$-simplex in a finite metric tree $(T,d)$. Currently, the domain of
the simplex gap function $\gamma_{D}$ is restricted to the surface of normalized load vectors $N_{q,t} \subset
\mathbb{R}_{+}^{q+t}$. We would like to extend the domain of definition of $\gamma_{D}$ to all of the open
set $\mathbb{R}_{+}^{q+t}$ in such a way that the extended simplex gap function (which we will denote $\gamma_{D}^{\times}$)
retains an accessible encoding of the geometry of the underlying tree $T$. We do this by ``formally'' adapting
the formulas of Theorem \ref{FORMULA}.

\begin{defn}\label{XGAP}
Let $(T,d)$ be a finite metric tree. Let $D=[a_{j};b_{i}]_{q,t}$ be a $(q,t)$-simplex in $T$.
The \textit{extended simplex gap function} $\gamma_{D}^{\times} : \mathbb{R}_{+}^{q+t} \rightarrow \mathbb{R}$
is defined as follows:
\begin{eqnarray*}
\gamma_{D}^{\times}(\vec{\omega})
& = & \sum\limits_{e \in E(T_{D})} \biggl\{ (\alpha_{L}(\vec{\omega},e) - \beta_{L}(\vec{\omega},e))^{2}
+ (\alpha_{R}(\vec{\omega},e) - \beta_{R}(\vec{\omega},e))^{2} \biggl\} \cdot \frac{|e|}{2}
\end{eqnarray*}
for all $\vec{\omega} =(m_{1}, \ldots, m_{q}, n_{1}, \ldots, n_{t}) \in \mathbb{R}_{+}^{q+t}$.
Notice that we have $\gamma_{D}^{\times}(\vec{\omega}) = \gamma_{D}(\vec{\omega})$
for all of the normalized load vectors $\vec{\omega} \in N_{q,t}$ on account of Theorem \ref{FORMULA}.
\end{defn}

\begin{notation}
In relation to Definition \ref{XGAP}, given an oriented edge $e \in E(T_{D})$, we will denote the ``$e$-term''
of the extended gap $\gamma_{D}^{\times}(\vec{\omega})$ by $\gamma_{D,e}^{\times}(\vec{\omega})$. That is:
\begin{eqnarray*}
\gamma_{D,e}^{\times}(\vec{\omega}) & = &
\frac{\bigl( (\alpha_{L}(\vec{\omega},e) - \beta_{L}(\vec{\omega},e))^{2} +
(\alpha_{R}(\vec{\omega},e) - \beta_{R}(\vec{\omega},e))^{2} \bigl) \cdot |e|}{2}.
\end{eqnarray*}
According to this notation, $\gamma_{D}^{\times}(\vec{\omega}) = \sum\limits_{e \in E(T_{D})} \gamma_{D,e}^{\times}(\vec{\omega})$
for each $\vec{\omega} \in \mathbb{R}_{+}^{q+t}$.

This notation is used in the proof of the next lemma. This lemma points out that provided the $(q,t)$-simplex $D$ is generically labeled,
the partial derivatives of the extended gap function $\gamma_{D}^{\times}$ pack together like Russian dolls when constrained to $N_{q,t}$,
the surface of normalized load vectors for $D$.
The lemma will help us compute $\min\limits_{\vec{\omega} \in N_{q,t}} \gamma_{D}^{\times}(\vec{\omega})$ in this (generically labeled)
setting.
\end{notation}

\begin{lem}\label{RDOLL}
Let $(T,d)$ be a finite metric tree.
Let $D=[a_{j};b_{i}]_{q,t}$ be a generically labeled $(q,t)$-simplex in $T$.
Let $\gamma_{D}^{\times}$ denote the extended gap function associated with the $(q,t)$-simplex $D$.
Then, for all oriented edges $e \in E(T_{D})$ and all normalized load vectors $\vec{\omega} \in N_{q,t}$,
we have the following relationships:
\begin{itemize}
\item[(a)] If $e=(a_{j},b_{i})$, then
\begin{eqnarray*}
\frac{\partial \gamma_{D}^{\times}}{\partial m_{j}} (\vec{\omega}) & = &
2 \left( \alpha_{L}(\vec{\omega},e) - \beta_{L}(\vec{\omega},e) \right) \cdot |e|
- \frac{\partial \gamma_{D}^{\times}}{\partial n_{i}} (\vec{\omega}).
\end{eqnarray*}

\item[(b)] If $e=(b_{i},a_{j})$, then
\begin{eqnarray*}
\frac{\partial \gamma_{D}^{\times}}{\partial n_{i}} (\vec{\omega}) & = &
2 \left( \beta_{L}(\vec{\omega},e) - \alpha_{L}(\vec{\omega},e) \right) \cdot |e|
- \frac{\partial \gamma_{D}^{\times}}{\partial m_{j}} (\vec{\omega}).
\end{eqnarray*}
\end{itemize}
\end{lem}

\begin{proof}
The proofs of Part (a) and Part (b) are very similar, so we will just concentrate on Part (a).
This requires us to consider a fixed  oriented edge $e \in E(T_{D})$ of the form $e=(a_{j},b_{i})$.

Suppose $f \not= e$ is some other oriented edge in the minimal subtree $T_{D}$. Then $f$ is either
to the left of $a_{j}$ or to the right of $b_{i}$. Let's assume, for arguments sake, that $f$ is to
the left of $a_{j}$. (The other case is entirely similar.) In this context we have both $a_{j}$ and $b_{i}$
on the right of $f$. That is, $j \in A_{R}(f)$ and $i \in B_{R}(f)$. (See Definition \ref{PARTSUM}.)
Consequently:
\begin{eqnarray*}
\frac{\partial \gamma_{D,f}^{\times}}{\partial m_{j}}(\vec{\omega}) & = & (\alpha_{R}(\vec{\omega},f) - \beta_{R}(\vec{\omega},f)) \cdot |f|,
\,\, {\rm{and}} \\
\frac{\partial \gamma_{D,f}^{\times}}{\partial n_{i}}(\vec{\omega}) & = & (\beta_{R}(\vec{\omega},f) - \alpha_{R}(\vec{\omega},f)) \cdot |f|,
\end{eqnarray*}
for all $\vec{\omega} \in \mathbb{R}_{+}^{q+t}$. By adding these two formulas we see that
\begin{eqnarray}\label{RDONE}
\Biggl( \frac{\partial}{\partial m_{j}} + \frac{\partial}{\partial n_{i}} \Biggl) \gamma_{D,f}^{\times}(\vec{\omega})& = & 0
\end{eqnarray}
for all oriented edges $f \not= e$ and all $\vec{\omega} \in \mathbb{R}_{+}^{q+t}$.

On the other hand, because $a_{j}$ and $b_{i}$ are on opposite sides of $e$, we see that
\begin{eqnarray*}
\frac{\partial \gamma_{D,e}^{\times}}{\partial m_{j}}(\vec{\omega}) & = &
\bigl( \alpha_{L}(\vec{\omega},e) - \beta_{L}(\vec{\omega},e) \bigl) \cdot |e|, \,\, {\rm{and}}\\
\frac{\partial \gamma_{D,e}^{\times}}{\partial n_{i}}(\vec{\omega}) & = &
\bigl( \beta_{R}(\vec{\omega},e) - \alpha_{R}(\vec{\omega},e) \bigl) \cdot |e|,
\end{eqnarray*}
for any $\vec{\omega} \in \mathbb{R}_{+}^{q+t}$. Therefore:
\begin{eqnarray*}
\Biggl( \frac{\partial}{\partial m_{j}} + \frac{\partial}{\partial n_{i}} \Biggl) \gamma_{D,e}^{\times} (\vec{\omega}) & = &
\bigl( (\alpha_{L}(\vec{\omega},e) - \alpha_{R}(\vec{\omega},e)) + (\beta_{R}(\vec{\omega},e) - \beta_{L}(\vec{\omega},e)) \bigl) \cdot |e|
\end{eqnarray*}
for all $\vec{\omega} \in \mathbb{R}_{+}^{q+t}$. If, in particular,  we evaluate this last formula for a normalized load vector
$\vec{\omega} \in N_{q,t}$, then we get the following simplifications:
\begin{eqnarray}\label{RDTWO}
\Biggl( \frac{\partial}{\partial m_{j}} + \frac{\partial}{\partial n_{i}} \Biggl) \gamma_{D,e}^{\times} (\vec{\omega})& = &
\bigl( \alpha_{L}(\vec{\omega},e) - (1 - \alpha_{L}(\vec{\omega},e) ) \bigl) \cdot |e| \nonumber \\
& ~ & + \bigl( (1-\beta_{L}(\vec{\omega},e)) - \beta_{L}(\vec{\omega},e) \bigl) \cdot |e| \nonumber \\
& ~ & ~ \nonumber \\
& = & 2 \bigl( \alpha_{L}(\vec{\omega},e) - \beta_{L}(\vec{\omega},e) \bigl) \cdot |e|.
\end{eqnarray}
The lemma now follows from equation (\ref{RDTWO}), which holds for the oriented edge $e$ on $N_{q,t}$, and equations (\ref{RDONE}),
which hold on $\mathbb{R}_{+}^{q+t}$ for all oriented edges $f \not= e$, by summing these equations over all such edges.
\end{proof}

Given a generically labeled $(q,t)$-simplex $D$ in a finite metric tree $(T,d)$ we now show how to
minimize the simplex gap $\gamma_{D} = \gamma_{D}(\vec{\omega})$ as a function of the normalized load vectors $\vec{\omega} \in N_{q,t}$.

\begin{thm}\label{GLMIN}
Let $(T,d)$ be a finite metric tree. Let $D=[a_{j};b_{i}]_{q,t}$ be a generically labeled $(q,t)$-simplex in $T$.
Let $\gamma_{D}^{\times}$ denoted the extended gap function associated with the simplex $D$. Let $N_{q,t} \subset
{\mathbb{R}}_{+}^{q+t}$ denoted the set of all normalized load vectors for $D$. Then:
\begin{eqnarray*}
\min\limits_{\vec{\omega} \in N_{q,t}} \gamma_{D}^{\times}(\vec{\omega}) & = &
\Biggl\{ \sum\limits_{e \in E(T_{D})} |e|^{-1} \Biggl\}^{-1}.
\end{eqnarray*}

In particular, if $d$ is just the ordinary path metric on $T$ (so that $|e|=1$ for all $e \in E(T_{D})$), then we get:
\begin{eqnarray*}
\min\limits_{\vec{\omega} \in N_{q,t}} \gamma_{D}^{\times} (\vec{\omega}) & = & \frac{1}{q+t-1}.
\end{eqnarray*}

Moreover, in general and in particular, the above minimums are attained if and only if $\vec{\omega} \in N_{q,t}$ is the
generic load vector $\vec{\omega}_{G}$ for $D$ which is assigned by the generic algorithm.
\end{thm}

\begin{proof}
The idea of the proof is to use Lagrange's (Multiplier) Theorem on a large scale. In relation to using this theorem,
note that the extended gap function $\gamma_{D}^{\times}$ is defined on an open set (namely, $\mathbb{R}_{+}^{q+t}$)
that contains the constraint surface $N_{q,t}$, which consists
of all normalized load vectors for $D$. We may assume (although it is
not strictly necessary) that the level zero vertices of the minimal subtree $T_{D}$ ($=D$, as sets) all have parity
$a$.

Accordingly, we introduce two Lagrange multipliers $\lambda_{1},\lambda_{2}$ and proceed to solve the system:
\begin{eqnarray}\label{SYSONE}
~ & ~ &
\left\{ \begin{array}{c}
\frac{\partial}{\partial m_{j}} \biggl( \gamma_{D}^{\times} (\vec{\omega}) -
\lambda_{1} \cdot \sum\limits_{j_{1}=1}^{q} m_{j_{1}} - \lambda_{2} \cdot
\sum\limits_{i_{1}=1}^{t} n_{i_{1}} \biggl) = 0, \,\, 1 \leq j \leq q \\
~ \\
\frac{\partial}{\partial n_{i}} \biggl( \gamma_{D}^{\times} (\vec{\omega}) -
\lambda_{1} \cdot \sum\limits_{j_{1}=1}^{q} m_{j_{1}} - \lambda_{2} \cdot
\sum\limits_{i_{1}=1}^{t} n_{i_{1}} \biggl) = 0, \,\, 1 \leq i \leq t
\end{array} \right.
\end{eqnarray}
subject to the two constraints imposed by the condition $\vec{\omega} \in N_{q,t}$.

Obviously we may rewrite the system of equations (\ref{SYSONE}) as:
\begin{eqnarray}\label{SYSTWO}
~ & ~ &
\left\{ \begin{array}{c}
\frac{\partial \gamma_{D}^{\times}}{\partial m_{j}} (\vec{\omega}) = \lambda_{1}, \,\, 1 \leq j \leq q\\
~ \\
\frac{\partial \gamma_{D}^{\times}}{\partial n_{i}} (\vec{\omega}) = \lambda_{2}, \,\, 1 \leq i \leq t\\
~\\
\vec{\omega} \in N_{q,t}. \hspace*{0.85in}
\end{array} \right.
\end{eqnarray}

Now consider an arbitrary oriented edge $e \in E(T_{D})$. If $e=(a_{j},b_{i})$ then system (\ref{SYSTWO})
in tandem with Lemma \ref{RDOLL} gives
\begin{eqnarray*}
\bigl( \alpha_{L}(\vec{\omega},e) - \beta_{L}(\vec{\omega},e) \bigl) \cdot |e| & = & \frac{\lambda_{1}+\lambda_{2}}{2}.
\end{eqnarray*}
Recalling that $\alpha_{L}(\vec{\omega},e) = \overline{\alpha}_{L}(\vec{\omega},e) + m_{j}$ then gives:
\begin{eqnarray*}
m_{j} & = & \beta_{L}(\vec{\omega},e) - \overline{\alpha}_{L}(\vec{\omega},e) + \frac{\lambda_{1}+\lambda_{2}}{2|e|}.
\end{eqnarray*}

On the other hand, if $e=(b_{i},a_{j})$, we (similarly) get:
\begin{eqnarray*}
n_{i} & = & \alpha_{L}(\vec{\omega},e) - \overline{\beta}_{L}(\vec{\omega},e) + \frac{\lambda_{1}+\lambda_{2}}{2|e|}.
\end{eqnarray*}

Hence the solution vector $\vec{\omega}=(m_{1}, \ldots, m_{q}, n_{1}, \ldots, n_{t}) \in N_{q,t}$ to system (\ref{SYSTWO})
satisfies the generic algorithm of Definition \ref{GENALG} with $\delta = (\lambda_{1}+\lambda_{2})/2$. In particular,
by applying Lemma \ref{GALEM} (b), we conclude that the solution vector $\vec{\omega} \in N_{q,t}$ is uniquely
determined and must be the generic load vector $\vec{\omega}_{G}$ assigned to the $(q,t)$-simplex $D$ by the generic algorithm.
Moreover, by Lemma \ref{GALEM} (b) and
Theorem \ref{FORMULA}, in conjunction with the computation in the latter part of the proof of Theorem \ref{UPPERBOUND}
(with $T$ replaced by $T_{D}$ in the obvious way), we conclude:
\[
\frac{\lambda_{1}+\lambda_{2}}{2} = \delta = \Biggl\{ \sum\limits_{e \in E(T_{D})} |e|^{-1} \Biggl\}^{-1} = \gamma_{D} (\vec{\omega}_{G}).
\]
Appealing to Lagrange's (Multiplier) Theorem completes the proof.
\end{proof}

Let $D(\vec{\omega})$ be a normalized $(q,t)$-simplex in a finite metric tree and let $T_{D} \subseteq T$ be the minimal
subtree generated by the vertices of $D$. If $D$ is not generically labeled there are two ways we can prune the minimal
subtree $T_{D}$ that lead (after a finite number of steps) to a generically labeled normalized
$(q^{\prime},t^{\prime})$-simplex $D_{\ast}(\vec{\omega}_{\ast})$ in a modified finite metric tree $(T_{\ast},d)$
with a smaller simplex gap: $\gamma_{D}(\vec{\omega}) > \gamma_{D_{\ast}}(\vec{\omega}_{\ast})$. These pruning
operations are described in the following lemma.

\begin{lem}\label{REDUCE}
Let $(T,d)$ be a finite metric tree.
Let $D(\vec{\omega})=[a_{j}(m_{j});b_{i}(n_{i})]_{q,t}$ be a normalized $(q,t)$-simplex in $T$.
Let $T_{D}$ denoted the minimal subtree of $T$ generated by the vertices of $D$.
Suppose $e_{\ast} = (x,y)$ is
an oriented edge in $T_{D}$ with one of the following two properties:
\begin{enumerate}
\item[(a)] $x,y \in D$ and $x,y$ have the same simplex parity, or

\item[(b)] $x \notin D$ or $y \notin D$.
\end{enumerate}
Form a new normalized $(q^{\prime},t^{\prime})$-simplex ${D_{\ast}}(\vec{{\omega}}_{\ast})$ and corresponding minimal subtree $T_{{D_{\ast}}}$ ---
within a modified tree $({T}_{\ast},d)$ --- by identifying vertex $x$ with vertex $y$ and by adding the simplex weights
associated with $x$ and $y$, if any. (In other words, to form ${T}_{\ast}$, delete the oriented edge ${e}_{\ast}$ from $T$ and
paste. And so on.) Then, recalling the more precise notation introduced after Definition \ref{PARTSUM}, we have:
\begin{eqnarray*}
\gamma_{{D_{\ast}}}(\vec{{\omega}}_{\ast}) & = &
\gamma_{D}(\vec{\omega}) - (\alpha_{L}(D, \vec{\omega}, {e}_{\ast}) - \beta_{L}(D, \vec{\omega}, {e}_{\ast}))^{2} \cdot |{e}_{\ast}| \\
& = & \sum\limits_{e \in E(T_{D}) \setminus \{ {e}_{\ast} \}} (\alpha_{L}(D, \vec{\omega}, e) - \beta_{L}(D, \vec{\omega}, e))^{2} \cdot |e|.
\end{eqnarray*}
In particular, we see that $\gamma_{D}(\vec{\omega}) > \gamma_{{D}_{\ast}}(\vec{{\omega}}_{\ast})$.
\end{lem}

\begin{proof}
Assume condition (a) or (b) holds. Let $e$ be an oriented edge in $T_{D}$ such that $e \not= {e}_{\ast}$.
Obviously $e$ is an edge in $T_{{D}_{\ast}}$ too. Moreover, all edges in $T_{{D}_{\ast}}$ arise this way. Checking
four simple cases shows that the left (and right) partition sums for $e$ are invariant under the identification
$x \equiv y$. That is to say, $\alpha_{L}(D,\vec{\omega},e) = \alpha_{L}({D}_{\ast}, \vec{{\omega}}_{\ast}, e)$ and
$\beta_{L}(D,\vec{\omega},e) = \beta_{L}({D}_{\ast}, \vec{{\omega}}_{\ast}, e)$.
[There are four cases because (a) or (b) might hold and because $e$ is either to the
left of $x$ or to the right of $y$. Whatever the case, when we delete the oriented edge ${e}_{\ast}$, no simplex weight
shifts from the left to the right of $e$, or \textit{vice versa}.] So $\gamma_{D}^{(e)}(\vec{\omega}) =
\gamma_{{D}_{\ast}}^{(e)}(\vec{\omega}_{\ast})$. Now
apply Theorem \ref{FORMULA} to get the formulas in the statement of the lemma.
\end{proof}

\begin{thm}\label{MAIN}
Let $(T,d)$ be a finite metric tree.
For all normalized $(q,t)$-simplexes
$D(\vec{\omega})=[a_{j}(m_{j});b_{i}(n_{i})]_{q,t}$ in $T$ we have:
\begin{eqnarray}\label{TWENTY}
\sum\limits_{j_{1} < j_{2}} m_{j_{1}}m_{j_{2}} d(a_{j_{1}},a_{j_{2}}) +
\sum\limits_{i_{1} < i_{2}} n_{i_{1}}n_{i_{2}} d(b_{i_{1}},b_{i_{2}}) +
\Biggl\{ \sum\limits_{e \in E(T)} |e|^{-1} \Biggl\}^{-1} & ~ & ~ \\
\leq \sum\limits_{i,j} m_{j}n_{i} d(a_{j},b_{i}). & ~ & ~ \nonumber
\end{eqnarray}

Moreover, we have equality in (\ref{TWENTY}) if and only if $D=T$ (as sets) and $D$ is both generically labeled and
generically weighted.
\end{thm}

\begin{proof}
If $D(\vec{\omega})$ is generically labeled, then (\ref{TWENTY}) holds by Theorem \ref{GLMIN} because
\[
\Biggl\{ \sum\limits_{e \in E(T)} |e|^{-1} \Biggl\}^{-1} \leq
\Biggl\{ \sum\limits_{e \in E(T_{D})} |e|^{-1} \Biggl\}^{-1} = \inf\limits_{N_{q,t}} \gamma_{D} \leq \gamma_{D}(\vec{\omega}).
\]

If $D(\vec{\omega})$ is not generically labeled, we may apply Lemma \ref{REDUCE} a finite number of times to produce a
possibly smaller normalized
$(q^{\prime},t^{\prime})$-simplex ${D}_{\ast}(\vec{{\omega}}_{\ast})$ in a modified tree $({T}_{\ast},d)$
that is generically labeled and which satisfies (by Lemma \ref{REDUCE} in the first instance and Theorem \ref{GLMIN} in the
second instance):
\begin{eqnarray*}
\gamma_{D}(\vec{\omega}) & > & \gamma_{{D}_{\ast}}(\vec{{\omega}}_{\ast}) \\
& \geq & \Biggl\{ \sum\limits_{e \in E(T_{{D}_{\ast}})} |e|^{-1} \Biggl\}^{-1} \\
& \geq & \Biggl\{ \sum\limits_{e \in E(T_{D})} |e|^{-1} \Biggl\}^{-1} \\
& \geq & \Biggl\{ \sum\limits_{e \in E(T)} |e|^{-1} \Biggl\}^{-1}.
\end{eqnarray*}

From these two cases --- generically labeled, or not --- we conclude that (\ref{TWENTY}) holds in general.
Moreover, the characterization of equality in (\ref{TWENTY}) is clear from the statement and proof of Theorem
\ref{GLMIN}, together with the observation that the minimum
\[
\min\limits_{E \subseteq E(T)} \Biggl\{ \sum\limits_{e \in E} |e|^{-1} \Biggl\}^{-1}
\]
is uniquely attained when $E=E(T)$.
\end{proof}

As an automatic corollary to Theorem \ref{MAIN} we can compute the $1$-negative type gap of any finite
metric tree \textit{exactly}.

\begin{cor}\label{GAPFORMULA}
Let $(T,d)$ be a finite metric tree.
Let $\Gamma_{T} = \inf\limits_{D(\vec{\omega})} \gamma_{D}(\vec{\omega})$ denote the
$1$-negative type gap of $(T,d)$. Then:
\[
\Gamma_{T} = \Biggl\{ \sum\limits_{e \in E(T)} |e|^{-1} \Biggl\}^{-1}.
\]
\end{cor}

Notice that the constant $\Gamma_{T}$ in Corollary \ref{GAPFORMULA} is independent of the internal geometry of the tree $T$
and depends only upon the unordered distribution of the tree's edge weights. By way of analogy;
the situation we are encountering in Corollary \ref{GAPFORMULA} is to be compared to having a
box of matches of unequal lengths. No matter how we construct a metric tree $T$ by using all
of the matches in the box, we invariably get the same value for the $1$-negative type gap $\Gamma_{T}$.
The same phenomenon applies to the inequalities (\ref{TWENTY}) of Theorem \ref{MAIN}: they are
independent of the particular finite metric tree's internal geometry. This seems remarkable.

It is also the case that the above formula for $\Gamma_{T}$ holds for any countable metric tree $(T,d)$.
Simply note that since trees are connected (by definition),
the minimal subtree generated by any finite
subset of a countable tree $T$ must be finite. This then allows one to invoke Corollary \ref{GAPFORMULA} and make a simple
limiting argument. No proof is therefore necessary for the following corollary.

\begin{cor}\label{COUNTABLE}
Let $(T,d)$ be a countable metric tree. Then the $1$-negative type gap $\Gamma_{T}$ of $(T,d)$ is given by the
formula
\[
\Gamma_{T} = \Biggl\{ \sum\limits_{e \in E(T)} |e|^{-1} \Biggl\}^{-1}
\]
where it is understood that $\Gamma_{T}$ is taken to be zero if the series in the parentheses diverges.
\end{cor}

Corollary \ref{COUNTABLE} makes it clear that, given any $\Gamma \geq 0$, we can construct a countable metric tree $(T,d)$
whose $1$-negative type gap $\Gamma_{T} = \Gamma$.
The simplest way to do this is to consider an internal node, denoted $0$,
surrounded by countably many leaves, denoted $n$ where $n \in \mathbb{N}$. Using Corollary \ref{COUNTABLE} we can then
drive the $1$-negative type gap of this star with $\aleph_{0}$
leaves by varying the edge weights $d(0,n)$, $n \in \mathbb{N}$,
accordingly. In summary, we have the following.

\begin{cor}
For each non negative real number $\Gamma \geq 0$ there exists a countable metric tree $(T,d)$ such that
the $1$-negative type gap $\Gamma_{T}$ of $(T,d)$ equals $\Gamma$.
\end{cor}

We now return to the context of finite metric trees as they are the primary objects of interest in this paper.
In particular, Theorem \ref{MAIN} and Corollary \ref{GAPFORMULA} are seen to be key results of this paper. The inequalities
(\ref{TWENTY}) of Theorem \ref{MAIN} can be rephrased using Theorem \ref{GRUNT} and a scaling argument as follows.

\begin{thm}\label{REFORM}
Let $(T,d)$ be a finite metric tree.
Then for all natural numbers $n \geq 2$, all finite subsets $\{ x_{1}, \ldots , x_{n} \} \subseteq T$, and all
choices of real numbers $\eta_{1}, \ldots , \eta_{n}$ with $\eta_{1} + \cdots + \eta_{n} = 0$, we have
\begin{eqnarray}\label{TWENTYONE}
\frac{\Gamma_{T}}{2} \cdot \Biggl( \sum\limits_{\ell=1}^{n} |\eta_{\ell}| \Biggl)^{2}
+ \sum\limits_{1 \leq i,j \leq n} d(x_{i},x_{j}) \eta_{i} \eta_{j} & \leq & 0
\end{eqnarray}
where $\Gamma_{T}=\Biggl\{ \sum\limits_{e \in E(T)} |e|^{-1} \Biggl\}^{-1}$.
\end{thm}

\begin{proof} Fix a subset $\{x_{1}, \ldots, x_{n} \} \subseteq T$ ($n \geq 2$), together with
real numbers $\eta_{1}, \ldots , \eta_{n}$ such that $\eta_{1} +
\cdots + \eta_{n} = 0$. Without any loss of generality we may assume
that $(\eta_{1}, \ldots, \eta_{n}) \not= (0, \ldots, 0)$. By
relabelling (if necessary) we may assume there exist natural numbers
$q, t \in \mathbb{N}$ such that $q+t=n$, $\eta_{1}, \ldots, \eta_{q}
\geq 0$, and $\eta_{q+1}, \ldots, \eta_{n} < 0$.

As $\sum\limits_{j=1}^{q} \eta_{j} = - \sum\limits_{k=q+1}^{n} \eta_{k}$ we may define
$\alpha = \sum\limits_{j=1}^{q} \eta_{j} =
- \sum\limits_{k=q+1}^{n} \eta_{k}
= \frac{1}{2} \sum\limits_{\ell=1}^{n} |\eta_{\ell}| > 0$.

For $1 \leq j \leq q$, set $a_{j} = x_{j}$ and
$m_{j} = \eta_{j}/\alpha$. And
for $1 \leq i \leq t$, set $b_{i} = x_{n-i+1}$ and
$n_{i} = - \eta_{n-i+1}/\alpha$.
By construction, $D=[a_{j}(m_{j});b_{i}(n_{i})]_{q,t}$ is a normalized $(q,t)$-simplex.
Moreover, arguing as in the proof of Theorem 2.4 with exponent $p=1$, we see that
\begin{align*}
& \frac{1}{\alpha^{2}} \cdot \sum\limits_{1 \leq i,j \leq n} d(x_{i},x_{j}) \eta_{i} \eta_{j} = \\
& 2 \Biggl(\sum\limits_{1 \leq j_{1} < j_{2} \leq q}
               m_{j_{1}}m_{j_{2}} d(a_{j_{i}}, a_{j_{2}})
   + \sum\limits_{1 \leq i_{1}<i_{2} \leq t}
        n_{i_{1}}n_{i_{2}} d(b_{i_{1}}, b_{i_{2}})
   - \sum\limits_{j,i = 1}^{n}
        m_{j}n_{i} d(a_{j},b_{i}) \Biggl).
\end{align*}
Hence by Theorem 4.12,
\begin{eqnarray*}
\sum\limits_{1 \leq i,j \leq n} d(x_{i},x_{j}) \eta_{i} \eta_{j}
& \leq & -2 \alpha^{2} \cdot \Biggl\{ \sum\limits_{e \in E(T)} |e|^{-1} \Biggl\}^{-1}
 =  - \frac{\Gamma_{T}}{2} \cdot \Biggl( \sum\limits_{\ell=1}^{n} |\eta_{\ell}| \Biggl)^{2}. \\
\end{eqnarray*}
\end{proof}

\begin{rem}\label{altchar}
Because the constant $\Gamma_{T}$ appearing on the left hand side of (\ref{TWENTYONE}) is maximal we see
that Theorem \ref{REFORM} (alternately, Theorem \ref{MAIN}) provides the \textit{optimal enhancement}
of the $1$-negative type inequalities for finite metric trees.
Moreover, it is clear from the proof of Theorem \ref{GRUNT} (and, particularly, the equality (\ref{n5})
given in that proof)
that one may characterize the case of equality in (\ref{TWENTYONE}) \textit{directly} in terms of
$\{x_{1}, \ldots, x_{n}\}$ and $(\eta_{1}, \ldots, \eta_{n})$. Although this characterization is visibly apparent,
we leave the precise formulation to the interested reader.
\end{rem}

\section{Applications of the Negative Type Gap}
In this section we determine some applications of the negative type gap of a finite metric space $(X,d)$.
The main point is that if $|X| < \infty$
and if the $p$-negative type gap $\Gamma_{X,p}$ of $(X,d)$ is positive, then $(X,d)$ must have
strict $s$-negative type for some $s > p$.
In such a way, the negative type gap provides a new technique for obtaining lower bounds on the
maximal $p$-negative of certain finite metric spaces. We will illustrate this technique in the case of finite metric trees,
and then complete this section by constructing some basic examples to make a few final technical points.

This is perhaps a good time to recall that $p$-negative type holds on closed intervals of the form $[0,\wp]$.
Specifically, if $(X,d)$ is a metric space (finite or otherwise), then $(X,d)$ has $p$-negative type for all
$p$ such that $0 \leq p \leq \wp$, where $\wp = \max \{ p_{\ast}: (X,d)$ has $p_{\ast}$-negative type$\}$.
See, for example, Wells and Williams \cite{WW}.

We mentioned the following theorem at the end of Section $2$. The estimate (\ref{NEWFOUR}) derived in the
proof of this theorem is of independent interest and we will refer back to it later in this section.

\begin{thm}\label{LOWERBOUNDS}
Let $(X,d)$ be a finite metric space with $|X| \geq 3$.
Assume $(X,d)$ has a positive $1$-negative type gap $\Gamma_{X} = \Gamma_{X,1} > 0$.
Then there exists an $\zeta > 0$ such that $(X,d)$ has strict $p$-negative type for all $p \in (1-\zeta,1+\zeta)$.
Moreover, $\zeta$ may be chosen so that it depends only upon $\eeg$ and the set of non zero distances in $(X,d)$.
\end{thm}

\begin{proof}
We may assume that the metric $d$ is not a positive multiple of the discrete metric on $X$. (Otherwise, $(X,d)$ has strict
$p$-negative type for all $p > 0$.) And we will let $n$ denote $|X|$,
the cardinality of $X$ (which is assumed to be at least three).
Our focus will be on determining the interval $[1,1+\zeta)$. (Arguing the interval $(1-\zeta,1]$ is entirely similar.)

It is helpful to begin with a simple estimate which will be used later in the proof.
Namely; if $b > 1$, $k \in \mathbb{N}$ and $\ve > 0$, then:
\begin{eqnarray}
b^{1+\ve} - b < \frac{\eeg}{2k} & \,\, \text{if and only if}\,\, &
\ve < \frac{\ln\left( 1 + \frac{\eeg}{2kb} \right)}{\ln b}. \label{ESTTWO}
\end{eqnarray}

Let $$\ees = \min\limits_{x \not= y} d(x,y)\,\, {\rm{and}}\,\, \eel = \max\limits_{x \not= y} d(x,y)$$ denote the shortest
and longest non zero distances in $(X,d)$. Our opening assumption on $d$ in this proof is that $\ees < \eel$.
By scaling (if necessary) we may further assume that $\ees \geq 1$.

Consider an arbitrary normalized $(q,t)$-simplex $\cutie$ in $X$. Note that both $q,t \leq n-1$ because $q+t \leq n$.
Given $p \geq 0$, we will use the following abbreviated notation throughout the remainder of this proof:
\begin{eqnarray*}
\eeL (p) & = & \sum\limits_{j_{1}<j_{2}} m_{j_{1}}m_{j_{2}}d(a_{j_{i}},a_{j_{2}})^{p}
+ \sum\limits_{i_{1}<i_{2}} n_{i_{1}}n_{i_{2}}d(b_{i_{1}},b_{i_{2}})^{p},\,\,{\rm{and}} \\
\eeR (p) & = & \sum\limits_{j,i} m_{j}n_{i}d(a_{j},b_{i})^{p}.
\end{eqnarray*}

Our $1$-negative type gap hypothesis on $(X,d)$, applied to the simplex $D(\vec{\omega})$, is therefore:
\begin{eqnarray}\label{NEWONE}
\eeL (1) + \eeg \leq \eeR (1).
\end{eqnarray}

The overall idea of the proof is to exploit the $1$-negative type gap $\eeg > 0$ to show
$$\eeL (1+\ve) < \eeL (1) + \frac{\eeg}{2}\,\, {\rm{and}}\,\, \eeR (1) - \frac{\eeg}{2} < \eeR (1+\ve)$$
provided $\ve > 0$ is sufficiently small.
If so, we then have $\eeL (1+\ve) < \eeR (1+\ve)$ by (\ref{NEWONE}), provided $\ve > 0$ is sufficiently small,
and (hence) the theorem follows.

In the current context we have $\eeR (1) < \eeR (1+\ve) < \eeR (1+\ve) + \frac{\eeg}{2}$
for all $\ve > 0$ because all of the
non zero distances in $(X,d)$ are at least one. Moreover, for all $\ell = d(x,y) \not= 0$ and all $\ve > 0$, we have
$\ell^{1+\ve} - \ell \leq \eel^{1+\ve} - \eel$.
This is because (for any fixed $\ve > 0$) the function $f(x) = x^{1+\ve} - x$ increases as
$x$ ($\geq 1$) increases.

Now, recalling that we need to show that
$\eeL (1+\ve) < \eeL (1) + \frac{\eeg}{2}$ for all sufficiently small $\ve > 0$, observe that we have
\begin{eqnarray}\label{NEWTWO}
\eeL (1+\ve) - \eeL (1) & = & \sum\limits_{j_{1}<j_{2}} m_{j_{1}}m_{j_{2}}
\left( d(a_{j_{1}},a_{j_{2}})^{1+\ve} - d(a_{j_{1}},a_{j_{2}}) \right) \\
& ~ & ~ \nonumber \\
& ~ & + \sum\limits_{i_{1}<i_{2}} n_{i_{1}}n_{i_{2}}
\left( d(b_{i_{1}},b_{i_{2}})^{1+\ve} - d(b_{i_{1}},b_{i_{2}}) \right) \nonumber \\
& \leq & (n-1)(n-2)(\eel^{1+\ve} - \eel) \nonumber
\end{eqnarray}
on the basis of the preceding comments. And, according to (\ref{ESTTWO}), we have:
\begin{eqnarray}\label{NEWTHREE}
\eel^{1+\ve} - \eel < \frac{\eeg}{2(n-1)(n-2)} & \,\, \text{iff} \,\, & \ve < \frac{\ln \left( 1 + \bigl\{
\frac{\eeg}{2 \eel (n-1)(n-2)} \bigl\}\right)}{\ln \eel}.
\end{eqnarray}
If we now set
\begin{eqnarray}\label{NEWFOUR}
\zeta & = & \frac{\ln \left( 1 + \bigl\{ \frac{\eeg}{2 \eel (n-1)(n-2)}\bigl\}\right)}{\ln \eel},
\end{eqnarray}
then it is clear that (\ref{NEWTWO}) and (\ref{NEWTHREE}) establish the theorem.
\end{proof}

Looking at the statement and proof of Theorem \ref{LOWERBOUNDS} it is clear that a more general theorem can be formulated.
This more general theorem, which we will now state, follows from simple modifications and adaptations of the proof of
Theorem \ref{LOWERBOUNDS}.

\begin{thm}\label{LOWERBOUNDS2}
Let $(X,d)$ be a finite metric space with $|X| \geq 3$.
Let $p_{1} \geq 0$. If the $p_{1}$-negative
type gap $\Gamma_{X,p_{1}} > 0$, then there exists an $\zeta > 0$ such that $(X,d)$ has strict
$p$-negative type for all $p \in (p_{1}-\zeta,p_{1}+\zeta)$. Moreover,
$\zeta$ may be chosen so that it depends only
upon $\Gamma_{X,p_{1}}$ and the set of non zero distances in $(X,d)$. Note, however, that in
the case $p_{1}=0$ one must naturally work with the interval $p \in (0,\zeta)$.
\end{thm}

We mentioned in Section $1$ that it is not known if the maximal $p$-negative type of a finite
metric space $(X,d)$ can be strict.
The following automatic corollary of Theorem \ref{LOWERBOUNDS2} provides some information on this open question.

\begin{cor}
Let $(X,d)$ be a finite metric space with $|X| \geq 3$. Let $\wp$ denote the maximal $p$-negative type of $(X,d)$.
If $(X,d)$ has strict $\wp$-negative type, then the $\wp$-negative type gap $\Gamma_{X,\wp}$ of $(X,d)$ must equal zero.
\end{cor}

Theorem \ref{LOWERBOUNDS} and Corollary \ref{GAPFORMULA} automatically imply the
following generalization of Corollary \ref{HJORTH}.
Recall that Corollary \ref{HJORTH} is due to Hjorth \textit{et al}.\ \cite{HLM}.

\begin{thm}\label{LOWER}
Let $(T,d)$ be a finite metric tree with $n = |T| \geq 3$. Then there exists an $\zeta > 0$
such that $(T,d)$ has strict $p$-negative type for
all $p \in (1-\zeta,1+\zeta)$. Moreover, $\zeta$ may be chosen so that it depends only upon the unordered distribution
of the tree's edge weights.
\end{thm}

Looking at Theorem \ref{LOWER} and referring back to the estimate (\ref{NEWFOUR}) in the proof of Theorem \ref{LOWERBOUNDS}, we can extract
the following interesting corollary. This corollary gives a lower bound on the maximal $p$-negative type of any finite tree $T$
endowed with the ordinary path metric. Importantly, these lower bounds depend only on $|T|$.

\begin{cor}
Let $T$ be a finite tree with $|T| = n \geq 3$. Let $\wp_{T}$ denote the maximal $p$-negative type of $(T,d)$. Then:
\[
\wp_{T} \geq 1 + \frac{\ln \bigl( 1 + \frac{1}{(n-1)^{3}(n-2)} \bigl)}{\ln (n-1)}.
\]
\end{cor}

\begin{proof}
Simply observe that, in the notation of the proof of Theorem \ref{LOWERBOUNDS}, we have $\ees = 1$, $\eel \leq n-1$,
and that $\eeg = \Gamma_{T} = 1/(n-1)$ by Corollary \ref{GAPFORMULA}, so we may apply (\ref{NEWFOUR}) to obtain
the stated lower bound on $\wp_{T}$. We should point out that in applying (\ref{NEWFOUR}) in this context we have
removed a factor of $2$ from the expression for $\zeta$. It is clear that this can always be done in
the proof of Theorem \ref{LOWERBOUNDS}.
\end{proof}

For certain classes of finite metric trees $(T,d)$, such as ``stars'', it is possible to compute the maximum of all $p$ such that $(T,d)$ has
$p$-negative type. Such examples may then be strung together to form further interesting metric trees (with sometimes pathological
properties) such as the ``infinite necklace'' which is described in Example \ref{STAR2}.

\begin{example}[A star with $n$ leaves]\label{STAR}
Let $n \geq 2$ be a natural number. Let $Y_{n}$ denote the unique tree with $n+1$ vertices and $n$ leaves.
In other words, $Y_{n}$ consists of an internal node, which we will denote $r_{n}$, surrounded by $n$ leaves. We
endow $Y_{n}$ with the ordinary path metric $d$. Consequently, there are only two non zero distances
in this tree; $1$ \& $2$. The following theorem computes the maximal $p$-negative type of $Y_{n}$.
\end{example}

\begin{thm}\label{MAXIMAL}
For all natural numbers $n \geq 2$, the maximal $p$-negative type $\wp_{n}$ of the metric tree $(Y_{n},d)$
is given by $$\wp_{n} = 1 + \frac{\ln \bigl( 1+ \frac{1}{n-1}\bigl)}{\ln 2}.$$
\end{thm}

\begin{proof}
Consider a normalized $(q,t)$-simplex $D = D(\vec{\omega}) = [a_{j}(m_{j});b_{i}(n_{i})]_{q,t}$ in $Y_{n}$.

If the internal node $r_{n} \not= a_{j},b_{i}$ for all $j$ and $i$ then the generalized roundness inequalities (\ref{TWO})
become:
\[
\sum\limits_{j_{1} < j_{2}} m_{j_{1}}m_{j_{2}} \cdot 2^{p} + \sum\limits_{i_{1} < i_{2}} n_{i_{1}}n_{i_{2}} \cdot 2^{p} \leq
\sum\limits_{j,i} m_{j}n_{i} \cdot 2^{p},
\]
and these obviously hold for any $p \geq 0$. So we may assume that the internal node $r_{n}$ of $Y_{n}$ is represented
in the normalized simplex $D$ without any loss of generality. Say, $r_{n} = b_{1}$.

Now suppose that $t \geq 2$. Form a modified normalized $(q,t-1)$-simplex ${D}_{\ast} = {D}_{\ast}(\vec{{\omega}}_{\ast})$
in $Y_{n}$ by replacing
the pair $b_{1}(n_{1}), b_{2}(n_{2})$ in $D$ with $b_{1}(n_{1}+n_{2})$. In other words, remove the vertex $b_{2}$
from $D$ and add its simplex weight $n_{2}$ to that of $b_{1}$. Consider an arbitrary $p \geq 0$.
Let $\Delta_{\mathfrak{L}}$ and $\Delta_{\mathfrak{R}}$ denote the net change in the left and the right sides of the generalized
roundness-$p$ inequality (\ref{TWO}) when we pass from the modified normalized simplex ${D}_{\ast}$ to the
original normalized simplex $D$. It is not hard to see that:
\begin{eqnarray*}
\Delta_{\mathfrak{L}} & = & n_{1}n_{2} + \sum\limits_{2 < i \leq t} n_{2}n_{i} \cdot (2^{p}-1) \leq n_{2}(1-n_{2}) \cdot (2^{p}-1),
\,\,\,{{\rm and}} \\
\Delta_{\mathfrak{R}} & = & \sum\limits_{j=1}^{q} m_{j}n_{2} \cdot (2^{p}-1) = n_{2} \cdot (2^{p}-1).
\end{eqnarray*}
Because $\Delta_{\mathfrak{L}} < \Delta_{\mathfrak{R}}$ it follows that if
$p \geq 0$ satisfies the generalized roundness-$p$ inequality (\ref{TWO}) for the modified normalized simplex ${D}_{\ast}$,
then $p$ must also satisfy the generalized roundness-$p$ inequality (\ref{TWO}) for the original normalized simplex $D$. Hence,
by applying this rationale a finite number of times (as necessary), we may assume that the normalized simplex $D$
is generically labeled. That is, $t=1$ and $b_{1}=r_{n}$.

Now consider an arbitrary $p$ for which $(Y_{n},d)$ has $p$-negative type. Referring to our now generically
labeled normalized simplex $D$ we see that $p$ must satisfy:
\[
\sum\limits_{1 \leq j_{1} < j_{2} \leq q} m_{j_{1}}m_{j_{2}} \cdot 2^{p} \leq \sum\limits_{j=1}^{q} m_{j} \cdot 1^{p} = 1.
\]
That is: $$\Biggl( 1 - \sum\limits_{j=1}^{q} m_{j}^{2} \Biggl) \cdot 2^{p} \leq 2.$$

But $\max \Bigl( 1 - \sum\limits_{j=1}^{q} m_{j}^{2} \Bigl) = 1 - \frac{1}{q}$, which is realized when each weight
$m_{j} = \frac{1}{q}$ (in which case $D$ is also generically weighted), and so $p$ must satisfy $( 1 - \frac{1}{q} ) \cdot 2^{p-1} \leq 1$.
In other words: $$p \leq 1 + \frac{\ln \bigl( 1+ \frac{1}{q-1}\bigl)}{\ln 2}.$$
The right hand side of this last expression minimizes when $q=n$
and this implies:
\[
\wp_{n} = 1 + \frac{\ln \bigl( 1+\frac{1}{n-1}\bigl)}{\ln 2}.
\]
\end{proof}

Using Example \ref{STAR} and Theorem \ref{MAXIMAL} we can construct an infinite metric tree that has
strict $1$-negative type but does not have $p$-negative for any $p>1$.

\begin{example}\label{STAR2}
We can form an infinite tree $Y$ as follows: for each natural number $n \geq 2$ connect $Y_{n}$ to $Y_{n+1}$
by introducing a new edge which connects the internal node $r_{n}$ of $Y_{n}$ to the internal node $r_{n+1}$ of $Y_{n+1}$.
Endow $Y$ with the ordinary path metric $d$.
\end{example}

\begin{thm}\label{STRICTMAX}
The infinite metric tree $(Y,d)$ described in Example \ref{STAR2} has strict $1$-negative type
but does not have $p$-negative type for any $p > 1$. Moreover, the $1$-negative type gap $\Gamma_{Y}=0$.
\end{thm}

\begin{proof} Each normalized $(q,t)$-simplex $D$ in $(Y,d)$ spans a minimal
subtree $T_{D}$ of $Y$ which is finite. By Corollary \ref{HJORTH}, $(T_{D},d)$ has strict $1$-negative type.
Therefore $(Y,d)$ has strict $1$-negative type.

For all $n$, $(Y_{n},d)$ is a subtree of $(Y,d)$ and by Theorem \ref{MAXIMAL} it has maximal $p$-negative type $\wp_{n} =
1+ \frac{\ln (1+ \frac{1}{n-1})}{\ln 2}$. As $n \rightarrow \infty$ we see that
$\wp_{n} \rightarrow 1^{+}$. Hence $(Y,d)$ does not have $p$-negative type for any $p > 1$.
Moreover, $\Gamma_{Y}=0$ by Corollary \ref{COUNTABLE}.
\end{proof}

\section{Acknowledgments}
We would like to thank John Burns, Joe Diestel, Stratos Prassidis and Meera Sitharam for making helpful
comments during the preparation of this paper. We are particularly indebted to Stratos Prassidis who had privately asked
the second named author about the negative type of finite metric trees in the first place. Moreover, the second
named author is very grateful for generous research grants from Canisius College, and for the provision of excellent
sabbatical facilities at the University of New South Wales.

\bibliographystyle{amsalpha}

\begin{thebibliography}{22222}
\bibitem [1]{AC} N. Ailon and M. Charikar, \textit{Fitting tree metrics: Hierarchical clustering and Phylogeny},
Proceedings of the 46th Annual IEEE Symposium on Foundations of Computer Science (2005), 73--82.

\bibitem [2]{BART} Y. Bartal, \textit{On approximating arbitrary metrics by tree metrics}, Proceedings of the 30th
Annual ACM Symposium on Theory of Computing (1998), 161--168.

\bibitem [3]{BL} Y. Benyamini and J. Lindenstrauss, \textsl{Geometric Nonlinear Functional Analysis (Volume 1)},
American Mathematical Society (Providence), American Mathematical Society Colloquium Publications \textbf{48} (2000), xi+1--488.

\bibitem [4]{BMW} J. Bourgain, V. Milman and H. Wolfson, \textit{On type of metric spaces}, Trans. Amer. Math. Soc.
\textbf{294} (1986), 295--317.

\bibitem [5]{BDK} J. Bretagnolle, D. Dacunha-Castelle and J. L. Krivine, \textit{Lois stables et espaces $L^{p}$},
Ann. Inst. H. Poincar\'{e} \textbf{2} (1966), 231--259.

\bibitem [6]{CAY} A. Cayley, \textit{On the theory of analytical forms called trees}, Philos. Mag. \textbf{13} (1857), 19--30;
Reprinted in ``Mathematical Papers,'' Cambridge \textbf{3} (1891), 242--246.

\bibitem [7]{CCGGP} M. Charikar, C. Chekuri, A. Goel, S. Guha and S. Plotkin, \textit{Approximating a finite
metric by a small number of tree metrics}, Proceedings of the 39th Annual IEEE Symposium on Foundations of Computer Science (1998),
379--388.

\bibitem [8]{DL} M. M. Deza and M. Laurent, \textsl{Geometry of Cuts and Metrics}, Springer-Verlag (Berlin, Heidelberg),
Algorithms and Combinatorics \textbf{15} (1997), xii+1--587.

\bibitem [9]{DGLY} A. N. Dranishnikov, G. Gong, V. Lafforgue and G. Yu, \textit{Uniform embeddings into Hilbert
space and a question of Gromov}, Can. Math. Bull. \textbf{45} (2002), 60--70.

\bibitem [10]{E1} P. Enflo, \textit{On the nonexistence of uniform homeomorphisms between $L_{p}$-spaces},
Ark. Mat. \textbf{8} (1968), 103--105.

\bibitem [11]{E2} P. Enflo, \textit{On a problem of Smirnov}, Ark. Mat. \textbf{8} (1969), 107--109.

\bibitem [12]{FRT} J. Fakcharoenphol, S. Rao and K. Talwar, \textit{A tight bound on approximating arbitrary metrics
by tree metrics}, Proceedings of the 35th annual ACM symposium on Theory of computing (2003), 448--455.

\bibitem [13]{G} M. Gromov, \textsl{Asymptotic invariants of infinite groups}, Cambridge University Press (Cambridge),
Proc. Symp. Sussex, 1991: II, London Math. Soc. Lecture Notes \textbf{182} (1993), vii+295.

\bibitem [14]{HKM} P. G. Hjorth, S. L. Kokkendorff and S. Markvorsen, \textit{Hyperbolic spaces are of
strictly negative type}, Proc. Amer. Math. Soc. \textbf{130} (2001), 175--181.

\bibitem [15]{HLM} P. Hjorth, P. Lison\v{e}k, S. Markvorsen and C. Thomassen, \textit{Finite metric spaces
of strictly negative type}, Linear Algebra Appl. \textbf{270} (1998), 255--273.

\bibitem [16]{J} M. Junge, \textit{Embeddings of non-commutative $L_{p}$-spaces into non-commutative $L_{1}$-spaces,
$1<p<2$}, GAFA, Geom. Funct. Anal. \textbf{10} (2000), 389--406.

\bibitem [17]{KK} A. Koldobsky and H. K\"{o}nig, \textit{Aspects of the isometric theory of Banach spaces},
Handbook of the Geometry of Banach Spaces (Volume 1), North-Holland, Amsterdam (2001), 899--939.

\bibitem [18]{LP} J. -F. Lafont and S. Prassidis, \textit{Roundness properties of groups}, Geom. Dedicata
\textbf{117} (2006), 137--160.

\bibitem [19]{LTW} C. J. Lennard, A. M. Tonge and A. Weston, \textit{Generalized roundness and negative type},
Mich. Math. J. \textbf{44} (1997), 37--45.

\bibitem [20]{LTW2} C. J. Lennard, A. M. Tonge and A. Weston, \textit{Roundness and metric type}, J. Math. Anal. Appl.
\textbf{252} (2000), 980--988.

\bibitem [21]{BM} B. Maurey, \textit{Type, cotype and $K$-convexity}, Handbook of the Geometry of Banach Spaces
(Volume 2), North-Holland, Amsterdam (2003), 1299--1332.

\bibitem [22]{MN} M. Mendel and A. Naor, \textit{Metric cotype}, Ann. Math. (to appear).

\bibitem [23]{M} K. Menger, \textit{Die Metrik des Hilbert-Raumes}, Akad. Wiss. Wien Abh. Math.-Natur. K1
\textbf{65} (1928), 159--160.

\bibitem [24]{NS} A. Naor and G. Schechtman, \textit{Remarks on non linear type and Pisier's inequality},
J. Reine Angew. Math. (Crelle's Journal) \textbf{552} (2002), 213--236.

\bibitem [25]{N} P. Nowak, \textit{Coarse embeddings of metric spaces into Banach spaces}, Proc. Amer. Math.
Soc. \textbf{133} (2005), 2589--2596.

\bibitem [26]{R} J. Roe, \textsl{Lectures on Coarse Geometry}, University Lecture Series \textbf{31},
Amer. Math. Soc., Providence, Rhode Island (2003), vii+175.

\bibitem [27]{S1} I. Schoenberg, \textit{Remarks to Maurice Frechet's article ``Sur la d\'{e}finition axiomatique
d'une classe d'espaces distanci\'{e}s vectoriellement applicable sur l'espace de Hilbert.''}, Ann. Math. \textbf{36}
(1935), 724--732.

\bibitem [28]{S2} I. Schoenberg, \textit{Metric spaces and positive definite functions}, Trans. Amer. Math. Soc.
\textbf{44} (1938), 522--536.

\bibitem [29]{SS} C. Semple and M. Steel, \textsl{Phylogenetics}, Oxford University Press (Oxford, New York),
Oxford Lecture Series in Mathematics and Its Applications, \textbf{24} (2003), xiii+1--256.

\bibitem [30]{WOS} G. Weber, L. Ohno-Machado and S. Shieber, \textit{Representation in stochastic search for
phylogenetic tree reconstruction}, J. Biomedical Informatics \textbf{39} (2006), 43--50.

\bibitem [31]{WW} J. H. Wells and L. R. Williams, \textsl{Embeddings and Extensions in Analysis}, Ergebnisse
der Mathematik und ihrer Grenzgebiete \textbf{84} (1975), vii+1--108.

\bibitem [32]{W} A. Weston, \textit{On the generalized roundness of finite metric spaces}, J. Math. Anal. Appl.
\textbf{192} (1995), 323--334.

\bibitem [33]{Y} G. Yu, \textit{The coarse Baum-Connes conjecture for spaces which admit a uniform embedding
into Hilbert space}, Invent. Math. \textbf{139} (2000), 201--240.

\end{thebibliography}

\end{document}